\documentclass{article}
\usepackage[T1]{fontenc}
\usepackage[cp1250]{inputenc}
\usepackage{amsthm}
\usepackage{amsmath}
\usepackage{amsfonts}
\usepackage{multicol}
\usepackage{epsfig}
\usepackage{verbatim}
\usepackage[affil-it]{authblk}
\usepackage[authoryear]{natbib}
\usepackage{subcaption}
\usepackage{amssymb}

\usepackage{hyperref}
\hypersetup{
	bookmarksnumbered,
	linkcolor=black,
	citecolor=black,
	colorlinks=true,
}%

\newtheorem{theorem}{Theorem}
\newtheorem{lemma}{Lemma}
\newtheorem{corollary}{Corollary}
\theoremstyle{definition}
\newtheorem{example}{Example}
\theoremstyle{plain}

\newtheorem{proposition}{Proposition}

\def\S{\mathfrak S}

\def\diag{\mathrm{diag}}
\def\Det{\mathrm{Det}}
\def\L{\mathcal L}
\def\R{\mathbb R}

\title{Optimal designs for treatment comparisons represented by graphs}
\author{Samuel Rosa}
\affil{Faculty of Mathematics, Physics and Informatics, Comenius University, Bratislava, Slovakia}
\date{\today} 

\begin{document}
	
\maketitle

\begin{abstract}
	Consider an experiment consisting of a set of independent trials for comparing a set of treatments. In each trial, one treatment is chosen and the mean response of the trial is equal to the effect of the chosen treatment. We examine the optimal approximate designs for the estimation of a system of treatment contrasts under such model. These approximate treatment designs can be used to provide optimal treatment proportions for designs in more general models with nuisance effects (e.g., time trend, effects of blocks).
	For any system of pairwise treatment comparisons, we propose to represent such system by a graph. In particular, we represent the treatment designs for these sets of contrasts by the inverses of the vertex weights in the corresponding graph $G$. We show that then the positive eigenvalues of the information matrix of a treatment design are inverse to the positive eigenvalues of the vertex-weighted Laplacian of $G$. Note that such representation of treatment designs differs from the well known graph representation of block designs, which are represented by edges. We provide a graph-theoretic interpretation of the $D$-, $A$- and $E$-optimality for estimating sets of pairwise comparisons; as well as some optimality results for both the systems of pairwise comparisons and the general systems of treatment contrasts. Moreover, we provide a class of 'symmetric' systems of treatment contrasts for which the uniform treatment design is optimal with respect to a wide range of optimality criteria.
\end{abstract}

\section{Introduction}

Consider an experiment in which $v$ treatments are to be tested (possibly under the presence of some nuisance effects) and suppose that the aim of the experiment is to estimate a set of treatment contrasts. Some common examples of experiments with treatment and nuisance effects are blocking experiments (e.g., \cite{MajumdarNotz}), two-way elimination of heterogeneity (i.e., row-column designs, e.g., \cite{Jacroux}) and experiments under the presence of time trend (e.g., \cite{AtkinsonDonev}).

The positions of any two treatments in the considered system of contrasts need not be symmetric. For example, in an experiment of comparing test treatments with a control, the control clearly plays a special role, whilst the test treatments in the system of contrasts are interchangeable. In such situations it may be suboptimal to consider equireplicated designs, i.e., designs in which each treatment, including the control, has the same number of replications. In the experiment of comparing treatments with a control, it can be beneficial to employ control in more trials than any of the other treatments. Indeed, e.g., when one considers the $A$-, $E$- or $MV$-optimality, the optimal block designs for comparing test treatments with a control are not equireplicated, with the control replicated more times than the other treatments, e.g., see \cite{GiovagnoliWynn}, \cite{MajumdarNotz}, \cite{Jacroux87}. In general, for a specified optimality criterion, the optimal numbers of treatment replications depend on the chosen set of contrasts.

In this paper, we study the optimal treatment replications. Thus, we examine a model, in which the response in a given trial is determined by the treatment chosen for the trial and a random error. Such model can be described as a zero-way elimination of heterogeneity, a one-way analysis of variance, or a completely randomized experiment for $v$ treatments (where pre-specified replications of these treatments are randomly assigned to a given set of experimental units, e.g., see \cite{StallingsMorgan}). In such a model, the design of the experiment consists of choosing the 'best' numbers of treatment replications. As such, we call these designs the treatment designs. We study optimal approximate designs; therefore, rather than the actual numbers of treatment replications, we consider the treatment proportions, i.e., the relative numbers of trials that are to be performed with the particular treatments.

In \cite{RosaHarman16} it was shown that in a general model with treatment effects and nuisance effects, the optimal proportions of treatment replications are vital in obtaining optimal designs. In fact, the attainment of optimal treatment proportions is a necessary condition of optimality, i.e., given an optimality criterion $\Phi$, any $\Phi$-optimal approximate design in the general model with additive nuisance effects must allocate the $\Phi$-optimal treatment proportions to the particular treatments. Thus, finding the optimal treatment proportions can be thought of as the first step in obtaining optimal approximate designs under the presence of the nuisance effects (see \cite{RosaHarman16}) - the optimal treatment designs determine the optimal (relative) numbers of treatment replications, which are then assigned to the particular nuisance conditions (e.g., blocks, time moments). 

Nearly optimal or highly efficient treatment replications can be obtained by the standard rounding methods (e.g., Chapter 12 in \cite{puk}) from optimal approximate treatment designs. The advantage of the approximate designs is in obtaining general optimality results with simpler form, unlike the optimal exact designs, which often break down to multiple special cases. Then, the results on optimal approximate designs can provide an insight into the qualitative behaviour of optimal exact designs. Indeed, the earlier mentioned optimal exact block designs for comparing treatments with a control from \cite{MajumdarNotz} and \cite{Jacroux87} tend to imply treatment proportions similar or equal to the optimal approximate treatment proportions given by \cite{GiovagnoliWynn} or \cite{RosaHarman16}.
\bigskip

We primarily focus on a common class of treatment contrasts in which the aim of the experiment is to estimate a system of pairwise comparisons of treatments. We propose to represent such systems by graphs: the treatments are represented by vertices and the edges connecting pairs of vertices represent the particular pairwise comparisons. That is a rather straightforward relationship, but the graph representation can be exploited further, extensively employing the graph Laplacians.

It turns out that the treatment designs can then be expressed as (inverses of) the vertex weights; for an analysis of vertex-weighted graphs see \cite{ChungLanglands}. Then, we show that the positive eigenvalues of the information matrix for a given treatment design are inverses of the positive eigenvalues of the Laplacian of the corresponding vertex-weighted graph. It follows that for the treatment designs, any optimality criterion that depends only on the eigenvalues of the information matrix (which include the well-known Kiefer's $\Phi_p$-optimality criteria, including the $D$-, $A$- and $E$-optimality) can be expressed using the eigenvalues of the graph Laplacian. This observation allows one to express the optimal design problem using the graph terminology and consequently obtain optimality results based on such representation.
We provide an interpretation of the $D$-, $A$- and $E$-optimality using the graph terminology. Moreover, we obtain optimality results for these criteria employing the graph representation, and some optimality results for general systems of contrasts.

Although earlier we emphasized that the uniform treatment proportions (i.e., equireplicated designs) need not be optimal, there are many systems of contrasts which imply the optimality of uniform treatment proportions for a wide range of optimality criteria.
By employing some graph properties, we obtain a class of sets of pairwise comparisons in which the uniform treatment design is optimal with respect to any orthogonally invariant information function. We then extend these results to general sets of treatment contrasts.

It is well known that block designs can be represented by graphs (see, e.g., \cite{CameronVanLint}, \cite{Cheng81}, \cite{BaileyCameron}), which may seem similar to the representation proposed here. However, these representations differ; in many aspects, these representations are opposite, which will be demonstrated later.


\subsection{Notation}

By $1_n$ and $0_n$, we denote the vectors of length $n$ of all ones and of all zeroes, respectively. The symbols $0_{m \times n}$ and $J_{m \times n}$ denote the $m \times n$ matrix of zeroes and the $m \times n$ matrix of ones, respectively. The matrices $I_n$ and $J_n$ are the $n \times n$ indentity matrix and the $n \times n$  matrix of ones, respectively. Let $x \in \R^n$, then  we denote by $\mathrm{diag}(x)$ the diagonal matrix with elements of $x$ on its diagonal; furthermore, by $x^p$, $p \in \R$, we mean the vector of component-wise powers with the convention that $0^{-1}=0$. By $\mathfrak{S}^n_+$ and $\mathfrak{S}^n_{++}$, we denote the set of $n\times n$ non-negative definite and positive definite matrices, respectively.
We denote the column and the null space of a square matrix $A$ by $\mathcal{C}(A)$ and $\mathcal{N}(A)$, respectively. The trace of a matrix $A$ is denoted by $\mathrm{tr}(A)$. 
For $A \in \mathfrak{S}^n_+$, we denote its eigenvalues as $\lambda_1(A) \geq \ldots \geq \lambda_n(A)$. For clarity, we will also use $\lambda_{\max}(A)=\lambda_1(A)$ and $\lambda_{\min}(A)=\lambda_n(A)$. By $A^-$ and $A^+$, we denote the generalized inverse of $A$ and the Moore-Penrose pseudoinverse of $A$, respectively. For any $A \in \S^n_+$, we define $\Det(A)$ as the pseudo determinant of $A$  (i.e., the product of all non-zero eigenvalues of $A$); furthermore, we define $A_{ij}$ as the matrix obtained from $A$ by deleting its $i$-th row and $j$-th column. 
	
\subsection{The Model}

We consider the model of zero-way elimination of heterogeneity
\begin{equation}\label{eModel1}
Y_i=\tau_{u(i)}+\varepsilon_i, \quad i=1,\dots,N,
\end{equation}
where $Y_i$ is the response in the $i$-th trial, $u(i) \in \{1,\ldots,v\}$ is the treatment chosen for the $i$-th trial, $\tau:=(\tau_1, \ldots, \tau_v)^T$ is the vector of treatment effects and $\varepsilon_1, \ldots, \varepsilon_N$ are i.i.d. random errors with $E(\varepsilon_i)=0$ and $D(\varepsilon_i) = \sigma^2 < \infty$ for all $i=1, \ldots, N$.

Suppose that the aim of the experiment is to estimate a system of $s$ treatment contrasts $Q^T\tau$. A contrast is a linear combination whose coefficients sum to zero, thus, the $v \times s$ coefficient matrix $Q$ satisfies $Q^T 1_v = 0_s$. We denote the elements of $Q$ as $q_{ij}$ and its columns as $q_1, \ldots, q_s$. We denote $r:=\mathrm{rank}(Q)$ and we say that $Q^T\tau$ is a full-rank system if $s \leq v$ and $r = s$, otherwise we say that $Q^T\tau$ is rank deficient. 
We will assume that the experimenters are interested in all treatments, i.e., no row of $Q$ is $0_s^T$.

An example of a full-rank system is the system of comparisons with control, where $Q^T= (-1_{v-1}, I_{v-1})$, which aims at estimating $\tau_2-\tau_1, \ldots, \tau_v-\tau_1$. The set of centered contrasts $\tau_i - \bar{\tau}$, $i=1, \ldots, v$, where $\bar{\tau} = \sum_i \tau_i /v$, with the coefficient matrix $Q=I_v - J_v/v$, is rank deficient.
\bigskip

An exact design in model \eqref{eModel1} is a function $\xi: \{1, \ldots, v\} \to \{0,1,\ldots,N\}$ which determines for each treatment the number of trials assigned to that treatment. In this paper, we will examine approximate designs, e.g., see \cite{Pazman86} and \cite{puk}. An approximate design in model \eqref{eModel1} is a function $w: \{1,\ldots,v\} \to [0,1]$ that satisfies $\sum_{i=1}^v w(i) = 1$. Thus, the approximate design specifies the treatment proportions (weights), i.e., for each treatment $i$, the value $w(i)$ determines the proportion of all trials that are performed with that treatment. 
For brevity, we refer to the approximate designs simply as designs, and we usually represent a design $w$ as a vector $w=(w_1, \ldots, w_v)^T$. Since the approximate designs in \eqref{eModel1} determine the treatment proportions, we alternatively call them treatment designs.

The system $Q^T\tau$ is estimable under $w$ if and only if $\mathcal{C}(Q) \subseteq \mathcal{C}(M(w))$, where $M(w) = \mathrm{diag}(w_1, \ldots, w_v) = \diag(w)$ is the moment matrix of $w$, see \cite{puk}. If $Q^T\tau$ is estimable under $w$, we say that $w$ is feasible for $Q^T\tau$. Since no row of $Q$ is $0_s^T$, a design $w$ is feasible in \eqref{eModel1} if and only if $w>0$. The information matrix of a feasible design $w>0$ for estimating a full-rank system $Q^T\tau$ is $N_Q(w) = (Q^TM^{-1}(w)Q)^{-1}$, see \cite{puk}.

In the case of rank deficient systems, the information matrix is not well defined, and thus, following \cite{puk}, we define for a feasible design $w>0$ the matrix $C_Q(w) = (Q^TM^{-1}(w)Q)^+$, which is an analogue to the information matrix $N_Q(w)$.
\bigskip

Note that instead of \eqref{eModel1}, we could consider the model 
\begin{equation}\label{eModelAlt}
Y_i = \mu + \tau_{u(i)} + \varepsilon_i,\quad i=1,\ldots,N,
\end{equation}
where $\mu$ is the constant term. Then, the moment matrix of a design $w$ is
$$M(w) = \begin{bmatrix}
M_{11}(w) & M_{12}(w) \\ M_{12}^T(w) & M_{22}(w),
\end{bmatrix} $$
where $M_{11}(w) = \diag(w)$, $M_{12}(w) = w$ and $M_{22}(w) = 1$. Let $M_\tau(w):= M_{11}(w) - M_{12}(w) M_{22}^-(w) M_{12}^T(w)$ be the Schur complement of $M_{22}(w)$ in $M(w)$. Then, $w$ is feasible for $Q^T\tau$ in \eqref{eModelAlt} if and only if $\mathcal{C}(Q) \subseteq \mathcal{C}(M_\tau(w))$ and in that case $N_Q(w) = (Q^T M_\tau^-(w) Q)^{-1}$; see, e.g., \cite{RosaHarman16}. It is easy to check that this can be simplified to: $w$ is feasible if and only if $w>0$ and in that case $N_Q(w) = (Q^TM_{11}^{-1}(w)Q)^{-1}$, which coincides with the feasibility condition and the information matrix in model \eqref{eModel1}, respectively. Similarly, $C_Q(w) = (Q^TM_{11}^{-1}(w)Q)^+$.  It follows that in \eqref{eModelAlt}, any optimality results are the same as in model \eqref{eModel1}, i.e., the results obtained in the following sections for model \eqref{eModel1} hold also for model \eqref{eModelAlt} with the constant term.
\bigskip

Let $\Phi: \mathfrak{S}^s_{+} \to \R$ be an optimality criterion. Then, we say that $w^*$ is $\Phi$-optimal for a full-rank system $Q^T\tau$ or for a rank deficient system $Q^T\tau$ if it maximizes $\Phi(N_Q(w))$ or $\Phi(C_Q(w))$ over all $w>0$, respectively. An optimality criterion $\Phi$ is an information function (see \cite{puk}) if it is positively homogeneous, superadditive, Loewner isotonic, concave and upper semicontinuous. We say that $\Phi$ is orthogonally invariant if $\Phi(H) = \Phi(UHU^T)$ for any orthogonal matrix $U$. Note that $\Phi(H)$ is orthogonally invariant if and only if it depends only on the eigenvalues of $H$; for more details, e.g., see \cite{Harman04}. For brevity, we will usually write $\Phi(w)$ instead of $\Phi(N_Q(w))$ or $\Phi(C_Q(w))$.

The well-known class of Kiefer's optimality criteria $\Phi_p$, $p \in [-\infty,0]$, are orthogonally invariant information functions. If $H \in \mathfrak{S}^s_{++}$, then
$$
\Phi_p(H)=
\begin{cases}
\; \Big(\frac{1}{s} \sum\limits_{j=1}^{s} \lambda_j^p(H) \Big)^{1/p}, & p \in (-\infty, 0), \\
\; \Big(\prod\limits_{j=1}^{s} \lambda_j(H) \Big)^{1/s}, & p=0, \\
\; \lambda_\mathrm{min}(H), & p=-\infty.
\end{cases}
$$
For $p=0, -1$ and $-\infty$, we obtain the criteria of $D$-, $A$- and $E$-optimality, respectively.
The rank deficient versions of the Kiefer's optimality criteria are defined on the positive eigenvalues of $C_Q(w)$ (see Section 8.18 of \cite{puk}), i.e., on $\lambda_1(C_Q(w)), \ldots, \lambda_r(C_Q(w))$, where $r = \mathrm{rank}(Q)$.

For easier interpretation, we will alternatively use the criteria which are equivalent to the $\Phi_p$ criteria in the sense of the implied ordering of designs, but are to be \emph{minimized}. Let $V_Q(w):=Q^TM^{-1}(w)Q$, let $Q^T\tau$ be a full-rank or rank deficient system with $\mathrm{rank}(Q)=r$ and let $w>0$. Then,
$$
\Psi_p(w)=
\begin{cases}
	\; \sum\limits_{j=1}^{r} \lambda_j^{-p}(V_Q(w)), & p \in (-\infty, 0), \\
	\; \prod\limits_{j=1}^{r} \lambda_j(V_Q(w)) , & p=0, \\
	\; \lambda_\mathrm{max}(V_Q(w)), & p=-\infty.
\end{cases}
$$
In particular, $\Psi_{-1}(w)=\mathrm{tr}(V_Q(w))$, and $\Psi_0(w)=\Det(V_Q(w))$. A design is said to be $\Psi_p$-optimal for $Q^T\tau$ if it minimizes $\Psi_p(w)$ over all $w>0$. Then, a design is $\Psi_p$-optimal for $Q^T\tau$ if and only if it is $\Phi_p$-optimal for $Q^T\tau$.

\section{Graph representation}\label{sGraph}

\subsection{Graph theory}\label{ssGT}

We provide some basic terminology of graph theory; for more details, e.g., see \cite{Diestel}, \cite{CvetkovicEA} or \cite{Bapat}.
A directed graph $G$ is a pair $G=(V,E)$, where $V$ is a finite set of vertices $V=\{1,\ldots,v\}$, and $E$ is a set of edges $E=\{e_1,\ldots,e_s\}$, which are ordered pairs of vertices $e_k = (i_k, j_k)$, $i_k, j_k \in V$ for all $k$. If $e=(i,j)$, the edge $e$ is directed from $i$ to $j$. We will not allow loops or multiple edges between two vertices.

We say that two vertices $i, j \in V$ are adjacent, denoted as $i \sim j$, if there exists an edge $(i,j) \in E$ or $(j,i) \in E$. An edge $e$ is incident with a vertex $i$ if there exists a vertex $j \in V$ such that $e=(i,j)$ or $e=(j,i)$.
We say that a vertex $i$ has degree $d_i$, where $d_i$ is the number of edges incident with $i$. If all edges incident with a vertex $i$ are directed towards $i$, we say that $i$ is a sink vertex; if all edges incident with $i$ are directed from $i$, we say that $i$ is a source vertex.

A graph $G$ can be characterized by certain matrices.
The adjacency matrix $A$ of $G$ is a $v \times v$ matrix with its rows and colums indexed by $V$ and satisfying 
$$
A_{i,j}=
\begin{cases}
\; 1, & i \sim j, \\
\; 0, & \text{otherwise}.
\end{cases}
$$
The incidence matrix $R$ of $G$ is a $v \times s$ matrix with its rows and columns indexed by $V$ and $E$, respectively, that satisfies
$$
R_{i,e}=
\begin{cases}
\; 1, & e = (i,j) \text{ for some } j \in V, \\
\; -1, & e = (j,i) \text{ for some } j \in V,  \\
\; 0 &\text{otherwise}.
\end{cases}
$$
We say that the Laplacian matrix $L$ of $G$ is the $v \times v$ matrix $L:=RR^T$, i.e.,
$$
L_{i,j}=
\begin{cases}
\; d_i, & i=j, \\
\; -1, &i \sim j, \\
\; 0, &\text{otherwise}.
\end{cases}
$$
Note that $L = D - A$, where $D=\diag(d_1,\ldots,d_v)$ and hence $L$ does not depend on the orientation of the edges.

To specify that $V$ is a set of vertices of $G$, we may write $V(G)$ instead of $V$, similarly for $E(G)$, $A(G)$, and so forth.
\bigskip

Following \cite{ChungLanglands}, we say that $G$ is a graph with vertex weights $\alpha$ if $G$ is a graph and $\alpha$ is a function $\alpha: V \to \mathbb{R}_+$; for simplicity, we will write $\alpha_i:=\alpha(i)$ and $\alpha=(\alpha_1,\ldots,\alpha_v)^T$. Then, the vertex-weighted Laplacian $\mathcal{L}_{\alpha}$ (see \cite{ChungLanglands}) is a $v \times v$ matrix with its rows and columns indexed by $V$, satisfying
$$
\mathcal{L}_{\alpha;i,j}=
\begin{cases}
\; d_i\alpha_i, & i=j \\
\; -\alpha^{1/2}_i\alpha^{1/2}_j, &i \sim j \\
\; 0, &\text{otherwise},
\end{cases}
$$
which can be represented as $\mathcal{L}_\alpha=\diag(\alpha^{1/2}) RR^T\diag(\alpha^{1/2})$.
A more common version of a weighted graph is one with edge weights, resulting in an edge-weighted Laplacian $\tilde{L}$, which differs from the vertex-weighted version, see, e.g., \cite{Merris94}, \cite{Merris95}. Note that the often used normalized Laplacian (see, e.g., \cite{Chung},  \cite{CvetkovicEA}) is a special case of the vertex-weighted Laplacian, with vertex weights equal to the inverses of the degrees of the vertices, $\alpha_i=d_i^{-1}$.
\bigskip

An automorphism of $G$ is a bijection $\pi: V \to V$ satisfying $(i,j) \in E$ if and only if $(\pi(i),\pi(j)) \in E$. Since $\pi$ is a bijection, it is in fact a permutation of the vertices. Consider the permutation of vertex labels given by a permutation $\pi$. By $P_\pi$ we denote the $v\times v$ permutation matrix given by $\pi$, i.e., $P\pi$ satisfies $P_\pi x = \pi(x)$ for any $x \in \R^v$. Then, under the relabelling given by $\pi$, $R$ changes to $P_\pi R$, $A$ changes to $P_\pi A P_\pi^T$, and the vector of vertex weights $\alpha$ changes to $P_\pi \alpha$. It follows that if $\pi$ is an automorphism, $P_\pi R = R$ and $P_\pi A P_\pi^T = A$; thus, $P_\pi RR^T P_\pi^T = RR^T$ and the vertex-weighted Laplacian $\L_\alpha$ changes to $P_\pi\diag(\alpha^{1/2}) P_\pi^T RR^T P_\pi \diag(\alpha^{1/2}) P_\pi^T = P_\pi \L_\alpha P_\pi^T$.

Let $\pi$ be a permutation on $V$. A cycle $c$ is a sequence of the form $(i, \pi(i), \ldots, \pi^{k-1}(i))$, where $i \in V$ and $k$ is the smallest number such that $\pi^k(i)=i$. For any $j \in V$, we say that the cycle $c=(i, \pi(i), \ldots, \pi^{k-1}(i))$ contains $j$, denoted as $j \in c$, if there exists $m \in \mathbb{N}$ such that $\pi^m(i)=j$.
Then, any permutation $\pi$ can be decomposed into its cycles $c_1, \ldots, c_K$. By a cyclic permutation $\pi$, we mean a permutation that consists of only one cycle, i.e., for any $i,j \in V$ there exists $m \in \mathbb{N}$ such that $\pi^m(i)=j$.
For example, the permutation $\pi_1(1)=2$, $\pi_1(2)=1$ and $\pi_1(3)=3$ can be expressed as $\pi_1=(1,2)(3)$ and the permutation $\pi_2(1)=2$, $\pi_2(2)=3$,  $\pi_2(3)=1$ is a cyclic permutation $\pi_2=(1,2,3)$.

\subsection{Representation of treatment proportions}

A common class of systems of contrasts for treatment comparisons are systems of pairwise comparisons $\tau_j - \tau_i$, which are obtained when all the columns $q_k$ of $Q$ satisfy $q_k^T\tau = \tau_{j_k} - \tau_{i_k}$ for some $i_k, j_k \in \{1,\ldots,v\}$. We propose to represent any such system by a graph $G=(V,E)$, where $V$ is the set of treatments and $e = (j,i) \in E$ if and only if $\tau_j - \tau_i$ is present in $Q^T\tau$. Intuitively, the representation is rather straightforward - the  vertices represent the treatments and the edges represent the pairwise comparisons: if there is a comparison $\tau_j-\tau_i$ in $Q^T\tau$, then the corresponding vertices are connected by an edge directed from $j$ to $i$.
Then, it is easy to see that the incidence matrix $R$ coincides with the coefficient matrix $Q$, i.e., $R = Q$. Therefore, any system of pairwise comparisons $Q^T\tau$ can be represented by a directed graph with the incidence matrix given by $Q$.

\begin{example}\label{exSpecialContrasts}
	Let $v=7$ and suppose that the aim of the experiment is to estimate the following treatment comparisons: $\tau_2-\tau_1$, $\tau_3-\tau_2$, $\tau_4-\tau_3$, $\tau_5-\tau_3$, $\tau_6-\tau_5$, $\tau_7-\tau_6$. Such system of contrasts was considered in \cite{Mead90}. We denote these contrasts as $Q_1^T\tau$ and in the following sections we will use them to demonstrate our results. The matrix $Q_1$ is of the form
	$$ 
	Q_1=\begin{bmatrix}
	-1 &  0 &  0 &  0 &  0 &  0 \\
	1 & -1 &  0 &  0 &  0 &  0 \\
	0 &  1 & -1 & -1 &  0 &  0 \\
	0 &  0 &  1 &  0 &  0 &  0 \\
	0 &  0 &  0 &  1 & -1 & -1 \\
	0 &  0 &  0 &  0 &  1 &  0 \\
	0 &  0 &  0 &  0 &  0 &  1
	\end{bmatrix}.
	$$
	The corresponding graph $G_1$ (the graph in Figure \ref{fSpecialContrasts}, disregarding the vertex weights) allows for a clear representation of the system of contrasts.
\end{example}

\begin{figure}[h]
	\centering
	\epsfig{file=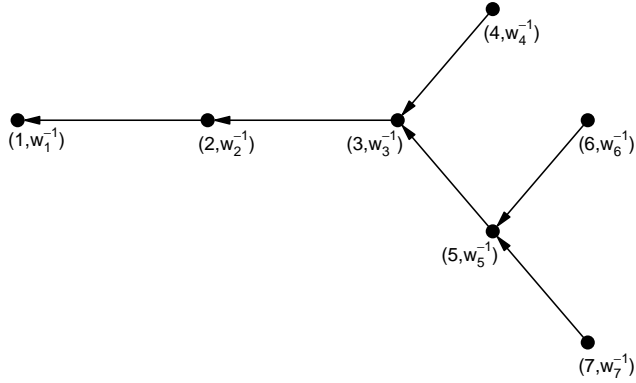, height=7cm}
	\vspace{-30pt}
	\caption{Graph representiation of the system of pairwise comparisons  $\tau_2-\tau_1$, $\tau_3-\tau_2$, $\tau_4-\tau_3$, $\tau_5-\tau_3$, $\tau_6-\tau_5$, $\tau_7-\tau_6$ and a design $w>0$ for this system of contrasts. The labels of the vertices are of the form $(i,\alpha_i)$, where $i$ is the index of the vertex and $\alpha_i=w_i^{-1}$ is the vertex weight.}
	\label{fSpecialContrasts}
\end{figure}

Note that a system of $v-1$ pairwise comparisons $Q^T\tau$ has full rank if and only if the corresponding graph is a tree (a connected graph without cycles; equivalently, a connected graph with $v-1$ edges). It follows from the fact that the corresponding graph is connected if and only if $\mathrm{rank}(L)=v-1$ (e.g., \cite{Mohar}). Furthermore, $\mathrm{rank}(Q) = \mathrm{rank}(QQ^T) = \mathrm{rank}(L)$ and the number of treatment contrasts is equal to the number of edges, i.e., $v-1$. Clearly, the graph $G_1$ in Example \ref{exSpecialContrasts} is a tree and thus the system of contrasts attains the full rank $v-1$.

The previous simple observation suggests that the Laplacian may be a useful tool in examining the systems of pairwise comparisons of treatments. Indeed, it is the case.
For a feasible design $w>0$ consider the vertex weights $\alpha$ given by the inverse values of $w$, i.e., $\alpha_i = w_i^{-1}$. Then, $\mathcal{L}_w = M^{-1/2}(w)QQ^TM^{-1/2}(w)$; recall that $M(w) = \diag(w)$. Note the slight abuse of notation, where $\L_w$ should in fact be expressed as $\L_\alpha$, where $\alpha=w^{-1}$ or as $\L_{w^{-1}}$.

Let $\Phi$ be an orthogonally invariant information function. Then, in the full-rank case, we may express $\Phi(w)$ as $\phi(\lambda_1(N_Q(w)), \ldots, \lambda_s(N_Q(w)))$ for some function $\phi$; in the rank deficient case, we have  $\Phi(w)=\phi(\lambda_1(C_Q(w)), \ldots, \lambda_r(C_Q(w)),0,\ldots,0)$. The following theorem shows that the value of $\Phi(w)$ is determined by the spectrum of the Laplacian $\mathcal{L}_w$ of the corresponding vertex-weighted graph.

\begin{theorem}\label{tEigLap}
	Let $\Phi$ be an orthogonally invariant information function and let $w$ be a feasible treatment design for estimating a system of pairwise comparisons $Q^T\tau$. Then, if $Q^T\tau$ is a full-rank system, $\Phi(w) = \phi(1/\lambda_s(\mathcal{L}_w), \ldots, 1/\lambda_1(\mathcal{L}_w))$ and if $Q^T\tau$ is rank deficient, $\Phi(w) = \phi(1/\lambda_r(\mathcal{L}_w), \ldots, 1/\lambda_1(\mathcal{L}_w),0,\ldots,0)$.
\end{theorem}

\begin{proof}
	Consider the full-rank case. Since $\Phi$ depends only on the eigenvalues of $N_Q$, we may write $\Phi(w)=\phi(\lambda_1(N_Q(w)), \ldots, \lambda_s(N_Q(w))) = \phi(1/\lambda_s(V_Q(w)), \ldots, 1/\lambda_1(V_Q(w)))$, because $\lambda_i(N_Q(w)) = 1/\lambda_{s-i+1}(V_Q(w))$. Moreover, it is well-known that a matrix $X^TX$ has the same positive eigenvalues, including multiplicities, as $XX^T$ (e.g., 6.54(c) in \cite{Seber}). Defining $X=M^{-1/2}(w)Q$ yields that $\Phi(w)=\phi(1/\lambda_s(\mathcal{L}_w), \ldots, 1/\lambda_1(\mathcal{L}_w))$, because $\mathcal{L}_w = XX^T$. Since the positive eigenvalues of $V_Q(w)$ are inverses of the positive eigenvalues of $C_Q(w)=V_Q^+(w)$, analogous results hold in the rank deficient case.
\end{proof}

Theorem \ref{tEigLap} shows that not only a system of pairwise comparisons can be expressed as a graph $G$, but also that a feasible treatment design $w>0$ can be represented by vertex weights on $G$ given by the \emph{inverse} values of $w$. For example, the system of contrasts for comparing test treatments with $g$ controls $\tau_j-\tau_i$, $i \in \{1,\ldots,g\}$, $j \in \{g+1,\ldots,v \}$ (e.g., see \cite{Majumdar86}) can be represented by a complete bipartite graph with partitions $\{1,\ldots,g\}$ and $\{g+1,\ldots,v\}$ (see Figure \ref{fCwC}), and the graph for all pariwise comparisons $\tau_j-\tau_i$, $j>i$ (e.g., \cite{BaileyCameron}), is a complete graph (see Figure \ref{fAPC}). The vertex weights implied by a design $w>0$ for the system of contrasts $Q_1^T\tau$ from Example \ref{exSpecialContrasts} are represented in Figure \ref{fSpecialContrasts}.

\begin{figure}[h]
	\centering
	\begin{subfigure}[b]{0.4\textwidth}
		\epsfig{file=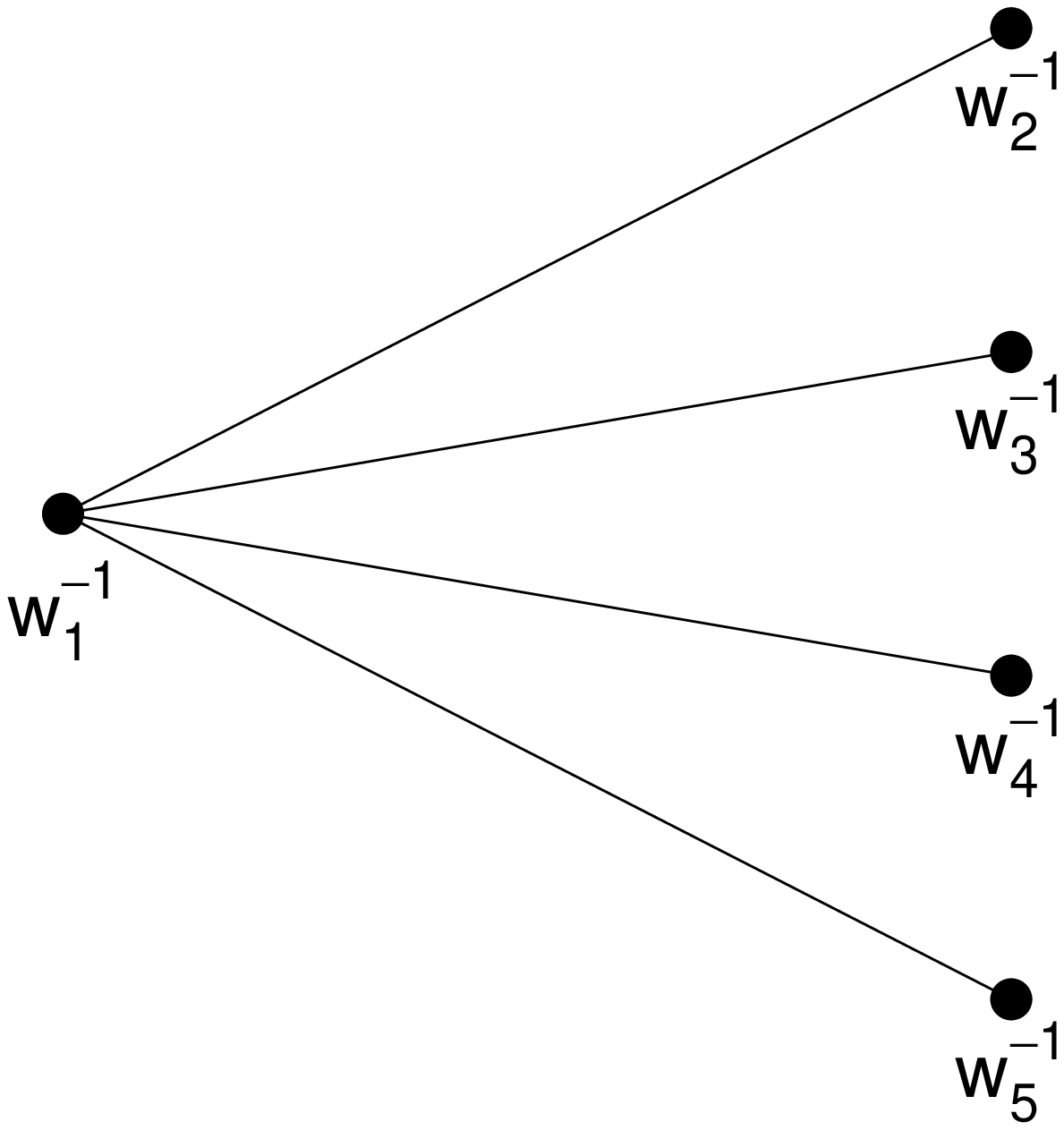, height=7cm}
		\vspace{-40pt}
		\caption{Comparison with one control}
		\label{fCwC}
	\end{subfigure}
	~ 
	\begin{subfigure}[b]{0.4\textwidth}
		\epsfig{file=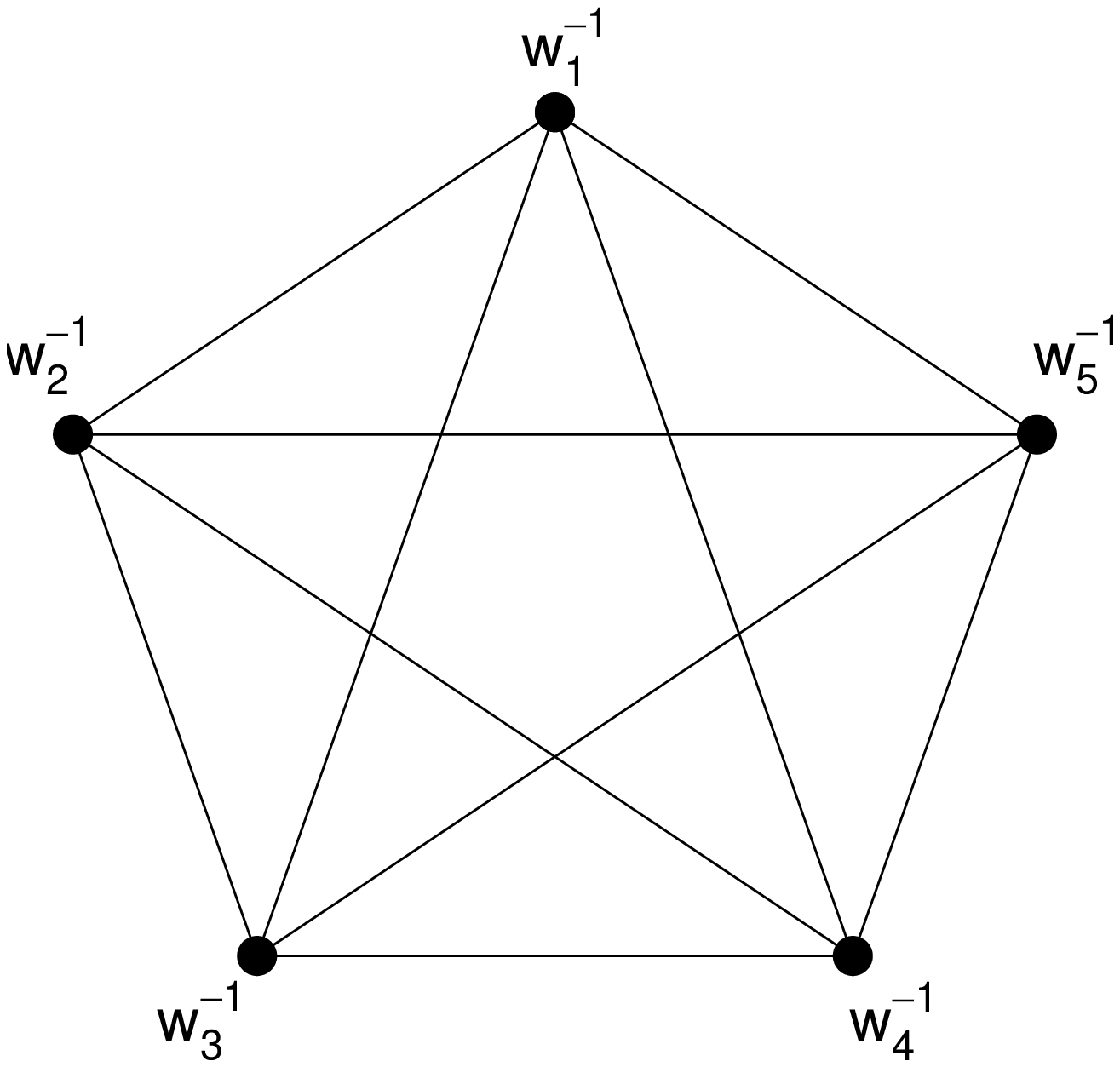, height=7cm}
		\vspace{-40pt}
		\caption{All pairwise comparisons}
		\label{fAPC}
	\end{subfigure}
	\caption{Graph representation of a treatment proportions design $w>0$ for selected systems of treatment contrasts. Since the Laplacian does not depend on the orientation of the edges, for clarity, the directions of the edges are suppressed. }\label{fGraphRep}
\end{figure}

From Theorem \ref{tEigLap} it follows that for a system of pairwise comparisons and an orthogonally invariant information function $\Phi$, the treatment design problem
$$\begin{aligned}
\max \quad & \Phi(N_Q(w)) \quad\text{or}\quad \Phi(C_Q(w))  \\
\text{s.t.} \quad & w>0,\quad \sum_{i=1}^v w_i = 1
\end{aligned}$$
can be expressed as
$$\begin{aligned}
\max \quad & \phi(1/\lambda_r(\L_w), \ldots, 1/\lambda_1(\L_w), 0, \ldots, 0) \\
\text{s.t.} \quad & w>0,\quad \sum_{i=1}^v w_i = 1,
\end{aligned}$$
where $\L_w$ is the Laplacian of the corresponding vertex-weighted graph.

Since $\alpha_i = w_i^{-1}$, the condition $\sum_i w_i = 1$ can be expressed as $\sum_i \alpha_i^{-1} = 1$ or $v/\sum_i \alpha_i^{-1} = v$, i.e., the harmonic mean of the vertex weights is equal to the number of vertices. Then, from the graph-theoretic point of view, the optimal design problem for a system of pairwise comparisons and an orthogonally invariant information function can be expressed as
$$\begin{aligned}
\max \quad & F(\L_\alpha)  \\
\text{s.t.} \quad & \alpha>0,\quad v/\sum_{i=1}^v \alpha_i^{-1} = v,
\end{aligned}$$
where $\L_\alpha$ is the Laplacian of a graph with vertex weights $\alpha$ and $F(\L_\alpha) := \phi(1/\lambda_r(\L_\alpha), \ldots,$ $1/\lambda_1(\L_\alpha), 0, \ldots, 0)$.
That is, the optimal design problem can be transformed to a problem of maximizing a given function $F$ defined on the eigenvalues of the Laplacian of a vertex-weighted graph over all vertex weights with a fixed harmonic mean.


Consider the Kiefer's $\Phi_p$-optimality criteria.
Using Theorem \ref{tEigLap}, $\Phi_p(w)$ may be expressed by employing the weighted Laplacian for both the full-rank and the the rank deficient case.

\begin{proposition}\label{pGraphKiefer}
Let $Q^T\tau$ be a system of pairwise comparisons with $\mathrm{rank}(Q)=r$ and let $w>0$ be a feasible design. Then, for any $p \in [-\infty,0]$, the value of the Kiefer's $\Phi_p$-optimality criterion can be expressed as
\begin{equation}\label{ePhipLaplacian}
\Phi_p(w)=
\begin{cases}
\; \Big(\frac{1}{r} \sum\limits_{j=1}^{r} \lambda_j^{-p}(\mathcal{L}_w) \Big)^{1/p}, & p \in (-\infty, 0), \\
\; \Big(\prod\limits_{j=1}^{r} \lambda_j(\mathcal{L}_w) \Big)^{-1/r}, & p=0, \\
\; 1/\lambda_{\max}(\mathcal{L}_w), & p=-\infty.
\end{cases}
\end{equation}
\end{proposition}
In particular, $\Psi_0(w)=\prod_{j \leq r}\lambda_j(\mathcal{L}_w)$, $\Psi_{-1}(w)=\mathrm{tr}(\L_w)$ and $\Psi_{-\infty}(w) = \lambda_{\max} (\L_w)$. 
\bigskip

Note that the graph characterization of designs in model \eqref{eModel1} significantly differs from the graph characterization of block designs, which are represented by edges.
In particular, it is well known that the exact block designs with blocks of size two can be represented by (undirected) graphs $G=(V,E)$, where the vertices represent the particular treatments and there is an edge between two vertices $i,j$ if and only if there is a block with treatments $i$ and $j$, e.g., see \cite{CameronVanLint}, \cite{BaileyCameron}. 

To demonstrate the differences between the graph representation proposed in this paper and the graph representation of block designs, we will consider \emph{approximate} block designs.
Approximate block designs $\xi$ with blocks of size two may be represented by graphs with \emph{edge weights} $c$, where the weight of edge $(i,j)$ is $c_{ij} = \sum_k \xi(i,k)\xi(j,k)$, where $\xi(i,k)$ is the value of $\xi$ for treatment $i$ and block $k$. Then, the edge-weighted Laplacian is $\tilde{L}_c = RCR^T $, where $C$ is a diagonal matrix with its rows and columns indexed by $E$, with diagonal elements $c_{ij}$. It turns out that for $c_{ij} = \sum_k \xi(i,k)\xi(j,k)$, the Laplacian satisfies $\tilde{L}_\xi=d^{-1}M_\tau(\xi)$, where $d$ is the number of blocks and $M_\tau(\xi)$ is the Schur complement of $M_{22}(\xi)$ in the moment matrix $M(\xi)$, which represents the amount of information on the vector of the treatment effects. Thus, if $\Phi$ is an information function, a $\Phi$-optimal design maximizes the value of $\Phi(\xi) = \Phi(M_\tau(\xi)) = \Phi(d^{-1}\tilde{L}_\xi) = d^{-1}\Phi(\tilde{L}_\xi)$, which is equivalent to maximizing $\Phi(\tilde{L}_\xi)$.

In particular, $\Phi_p(\xi)$ can be expressed as
$$
\Phi_p(\xi)=
\begin{cases}
\; d^{-1}\Big(\frac{1}{v-1} \sum\limits_{j=1}^{v-1} \lambda_j^p(\tilde{L}_\xi) \Big)^{1/p}, & p \in (-\infty, 0), \\
\; d^{-1}\Big(\prod\limits_{j=1}^{v-1} \lambda_j(\tilde{L}_\xi) \Big)^{1/(v-1)}, & p=0, \\
\; d^{-1}\lambda_{v-1}(\tilde{L}_\xi), & p=-\infty,
\end{cases}
$$
which are, in a sense, opposite problems to the optimality of treatment proportions $w$, compare with \eqref{ePhipLaplacian}. Particularly, $\Psi_0(w) = \prod_{j \leq v-1} \lambda_j^{-1}(\tilde{L}_\xi)$, $\Psi_{-1}(w)=\mathrm{tr}(\tilde{L}_\xi^{-1})$ and $\Psi_{-\infty}(w) = \lambda_{v-1}^{-1}(\tilde{L}_\xi)$.  The contrast between the criteria for block designs and for treatment designs intuitively follows from the fact that the graphs for treatment designs $w$ are weighted by the \emph{inverse} values of $w$.

To summarize, the approximate block designs are represented by the edge weights, and the optimality criterion is calculated using the eigenvalues of the edge-weighted Laplacian; whereas the treatment designs are represented by the inverses of the vertex weights, the edges are fixed: they are specified by the system of contrasts, and the optimality criterion is calculated using the inverses of the eigenvalues of the vertex-weighted Laplacian.

\subsection{Experiments on graphs}

The relationship between treatment designs and vertex weighted graphs (and especially the representation by inverse values of the vertex weights) can be derived in a slightly different manner, because it naturally arises from considering experiments on graphs.
Let $G=(V,E)$ be a graph with $|V|=v$ and consider the following experiment. In each vertex $i$, we carry out $n_i$ trials; each vertex $i$ implying a particular mean response $\tau_i$, and the responses are independent, i.e.,
$$
Y_{it} = \tau_i + \varepsilon_{it}, \quad t=1,\ldots,n_i, \quad i=1,\ldots,v.
$$
where  $\tau=(\tau_1, \ldots, \tau_v)^T$ are the mean responses in vertices $1, \ldots, v$, $E(\varepsilon_{it})=0$, $\mathrm{Var}(\varepsilon_{it})=\sigma^2$ and $\varepsilon_1,\ldots, \varepsilon_v$ are i.i.d. Suppose that an edge $(i,j) \in E$ represents that we are interested in comparing the mean responses in vertices $i$ and $j$, $\tau_i-\tau_j$. The mean response $\tau_i$ can be estimated by the mean of the responses in vertex $i$, and the comparison of mean responses given by an edge can be performed by comparing means of the responses.  Let $\bar{Y_i} = \sum_t Y_{it}/n_i$ be the least squares estimator (LSE) of $\tau_i$, with $\mathrm{Var}(\bar{Y_i})=\sigma^2/n_i$. Then, for $(i,j) \in E$ let $Z_{ij} =\bar{Y_i} - \bar{Y_j}$ be the LSE for $\tau_i - \tau_j$, with
$$\mathrm{Cov}(Z_{ij},Z_{k\ell})=\begin{cases}
-\sigma^2n_j^{-1} &\text{if } k=j, \\
\sigma^2n_j^{-1} &\text{if } \ell=j, \\
\sigma^2(n_i^{-1} + n_j^{-1}) &\text{if } i=k, j=\ell, \\
0 &\text{otherwise}.
\end{cases}$$
It follows that $\mathrm{Var}(Z)=\sigma^2 R^T \diag(n_1^{-1}, \ldots, n_v^{-1}) R$, where $Z$ is the vector of random variables $Z_{ij}$ indexed by $E$, and $R$ is the incidence matrix of $G$. Therefore, for any directed graph $G$ the matrix $R^T \diag(n_1^{-1}, \ldots, n_v^{-1}) R$ can be thought of as the variance matrix of $G$.
Clearly, $R^T \diag(n_1^{-1}, \ldots, n_v^{-1}) R$ has the same positive eigenvalues as the vertex weighted Laplacian $\diag(n^{-1/2})RR^T\diag(n^{-1/2})$ of $G$ with vertex weights $n_i^{-1}$. It is thus natural that for a design $w>0$, the corresponding graph is weighted by the inverses of the design values.

\section{Optimality}\label{sOptimal}

\subsection{Permutation of Treatments}\label{ssPerm}
To prove certain optimality results, we will consider treatment permutations; thus, we briefly examine such transformations.
Let $w$ be a design, let $\pi$ be a permutation of $\{1,\ldots,v\}$ and let $P_\pi$ be the permutation matrix corresponding to $\pi$. Then, we denote the design obtained from $w$ by $\pi$-permutation of the treatment labels as $P_\pi w$. Then, $M(P_\pi w) = P_\pi M(w)P_\pi^T$, $V_Q(P_\pi w) = Q^T P_\pi M^{-1}(w)P_\pi^T Q$, and $N_Q(P_\pi w)$ and $C_Q(P_\pi w)$ change analogously.

As noted in Section \ref{ssGT}, if $\pi$ is an automorphism of the graph $G$, the incidence matrix does not change under $\pi$ and for a general vertex weight function $\alpha$, the vertex-weighted Laplacian $\L_\alpha$ changes to $P_\pi \L_\alpha P_\pi^T$. In particular, for the weights implied by a design $w$, the vertex-weighted Laplacian $\L_w$ changes to $\L_{P w} = P_\pi \L_w P_\pi^T$. It follows that if $\pi$ is an automorphism of the corresponding graph $G$, the Laplacians for $w$ and for $P_\pi w$ are orthogonally similar.

\subsection{$D$-optimality}
In \cite{ChungLanglands}, the authors provide an interpretation of the coefficients of the characteristic polynomial of the Laplacian of a vertex-weighted graph
$$\det(\lambda I - \L_\alpha) = \sum_{k=0}^{v} (-1)^k c_k \lambda^{v-k}(\L_\alpha) ,$$
with $c_0=1$, and since $\L_\alpha$ is always singular, $c_n=0$.
In particular, they obtain that $c_{v-1}=\prod_{j \leq v-1} \lambda_j(\L_\alpha) = \kappa(G)$, 
where $\kappa(G)$ is the sum of weights of all rooted directed spanning trees, which is a generalization of the well-known Matrix-tree theorem (e.g., \cite{Mohar}) for the vertex-weighted graphs. Thus, $\Psi_0(w) = \kappa(G)$ if $r=v-1$.

For the formal definition of $\kappa(G)$, let $k \in \{1,\ldots,v-1\}$ and let $S \subseteq V$, $|S|=k$, and let $X$ be a subset of $v-k$ edges. If $(V(G),X)$ is a spanning forest (i.e., a subgraph of $G$ containing no cycle, disregarding the orientations of the edges) and each of the subtrees contains exactly one vertex in $S$, the rooted spanning forest $X_S$ is obtained by orienting all edges of $X$ towards $S$. Furthermore, define the weight of $X_S$ as $\omega(X_S) := \prod_{(i,j) \in E(X_S)} \alpha_j$; $\kappa_S:= \sum_{X_S} \omega(X_S)$ is the total weight of all spanning forests induced by $S$, and $\kappa_k(G) := \sum_{S:|S|=k} \kappa_S(G)$ is the sum of weights of all rooted directed spanning forests with $k$ roots. Then, $\kappa(G)=\kappa_1(G)$. \cite{ChungLanglands} obtain that
$c_k = \kappa_{v-k}(G)$ for all $k=1,\ldots,v-1$.

Consider the vertex weights given by a feasible design $w$, $\alpha=w^{-1}$.
It is well-known that the coefficients of the characteristic polynomial satisfy 
$$
c_k = \sum_{i_1 < \ldots < i_k} \lambda_{i_1}(\L_w) \cdot \ldots \cdot \lambda_{i_k}(\L_w)
$$
and therefore, if $\mathrm{rank}(Q)=v-1$, only the smallest eigenvale of $\L_w$ is equal to zero and $\kappa(G) = c_{v-1} = \prod_{i\leq v-1} \lambda_i(\L_w) = \Psi_0(w)$, which we stated earlier. If $\mathrm{rank}(Q)=r\leq v-1$, the $v-r$ smallest eigenvalues of $\L_w$ are zero and we obtain that $\kappa_{v-r}(G) = c_{r} = \prod_{i\leq r} \lambda_i(\L_w) = \Psi_0(w)$.

\begin{proposition}\label{pDopt}
	Let $Q^T\tau$ be a system of pairwise comparisons with $\mathrm{rank}(Q)=r$, let $w>0$ be a feasible design and let $G$ be the corresponding vertex-weighted graph. Then, $\Psi_0(w) = \kappa_{v-r}(G)$ and $\Phi_0(w)=1/\kappa_{v-r}(G)$.
\end{proposition}

Proposition \ref{pDopt} states that the value of $\Psi_0(w)$ is equal to the total weight of all rooted directed spanning forests with $r$ roots. In the most usual case, with $\mathrm{rank}(Q)=v-1$, it is the total weight of all directed rooted spanning trees. It follows that a $D$-optimal design minimizes the total weight of rooted directed spanning forests with $v-r$ roots. Compare with $D$-optimal block designs, which \emph{maximize} the number of spanning trees of a graph, see, e.g., \cite{Cheng81}, \cite{Bailey09}.

The $D$-optimality criterion is of limited interest to us if $s=v-1$ and $Q$ is a full-rank system, because it is well-known that in such case, the uniform design $\bar{w}=1_v/v$ is always $D$-optimal, which can be proved by a reparametrization of the system of interest. Such proposition can be generalized for any system of $s\geq v-1$ contrasts of rank $v-1$. Of course, for the number of contrasts $s>v-1$ we consider the rank-deficient version of $D$-optimality, $\Psi_0(w) = \prod_j \lambda_j(V_Q(w))$, where the product is over all positive eigenvalues of $V_Q(w)$. Thus, $\Psi_0(w)$ is the pseudo-determinant of $V_Q(w)$ (see \cite{Knill}), which is the product of all non-zero eigenvalues of $V_Q(w)$. The reparametrization argument cannot be easily replicated in such a case, because the pseudo-determinant does not satisfy $\Det(AB)=\Det(A)\Det(B)$ in general. First, we provide a technical lemma that states that if a symmetric matrix $A$ is in the class of matrices satisfying $A1_v = 0_v$, all its first minors are the same, up to a change of sign.

\begin{lemma}\label{lMinors}
	Let $A$ be a $v \times v$ symmetric matrix satisfying $A1_v = 0_v$. Then, $\det(A_{ij})=(-1)^{i+j}\det(A_{11})$ for all $i,j \in \{1,\ldots,v\}$.
\end{lemma}

\begin{proof}
	The Lemma follows from the well known fact that for any matrix $A$ satisfying the conditions of this Lemma, the cofactors of any two elements of $A$ are equal (e.g., see Lemma 4.2 in \cite{Bapat}).
\end{proof}

Now, we may formulate the optimality of the uniform design.

\begin{theorem}
	Let $Q^T\tau$ be a system of $s\geq v-1$ contrasts of rank $v-1$. Then, the uniform design $\bar{w}=1_v/v$ is $\Psi_0$-optimal for estimating $Q^T\tau$.
\end{theorem}

\begin{proof}
	Let $w>0$, let $\pi$ be a permutation of treatments and let $P:=P_\pi$. The moment matrix of $Pw$ is $M(Pw)=PM(w)P^T$ and therefore, $V_Q(Pw)=Q^TPM^{-1}(w)P^TQ$. We will use the facts provided in \cite{Knill} that for any two $a \times b$ matrices $F$, $G$ the following hold: $\Det(F^TG)=\Det(FG^T)$; and $\Det(F^TG) = \sum_X \det(F_X)\det(G_X)$, where the sum is over all $k \times k$ sub matrix masks $X$ of $F$ and $F_X$, $G_X$ are the corresponding submatrices, where $(-1)^k c_k\lambda^{b-k}$ is the smallest order entry in the characteristic polynomial of the $b \times b$ matrix $F^TG$. Then,
	$$
	\Psi_0(Pw) = \Det(V_Q(Pw)) = \Det(QQ^T P M^{-1}(w)P^T ) = \sum_X \det(QQ^T)_X \det(P M^{-1}(w)P^T)_X,
	$$
	where the sum is over all $(v-1) \times (v-1)$ submatrices, because $Q$ has rank $v-1$. Such sum may be expressed as $\sum_{i,j} \det(QQ^T)_{ij} \det(P M^{-1}(w)P^T))_{ij}$, where the sum is over all $i,j \in \{1,\ldots,v\}$. Because $\det(P M^{-1}(w)P^T)_{ij} = 0$ if $i \neq j$, we obtain 
	$$\Psi_0(Pw) =\sum_{i} \det(QQ^T)_{ii} \det(P M^{-1}(w)P^T)_{ii}.$$
	
	Since $Q$ is a matrix of contrasts, we have $1_v^TQ = 0_s^T$ and thus $QQ^T$ satisfies the conditions of Lemma \ref{lMinors}. It follows that $\det(QQ^T)_{ii} = \det(QQ^T)_{11}$; moreover, $\det(P M^{-1}(w)P^T)_{ii}$ is the product of all the $w_j^{-1}$-s except $w_{\pi^{-1}(i)}^{-1}$. Therefore, 
	$$\begin{aligned}
	\Psi_0(Pw) 
	&= \det(QQ^T)_{11} \sum_{i} \det(P M^{-1}(w)P^T)_{ii} 
	= \det(QQ^T)_{11} \sum_i \prod_{j \neq \pi^{-1}(i)} w_j^{-1} \\
	&=  \det(QQ^T)_{11} \sum_i \prod_{j \neq i} w_j^{-1} 
	= \det(QQ^T)_{11} \sum_i \det(M^{-1}(w))_{ii}= \Psi_0(w).
	\end{aligned}$$
	Hence, $\Phi_0(Pw)=\Phi_0(w)$, which yields
	$$
	\Phi_0(\bar{w}) 
	= \Phi_0(\frac{1}{v!}\sum_\pi P_\pi w) 
	\geq \frac{1}{v!} \sum_\pi \Phi_0(P_\pi w) = \Phi_0(w),
	$$
	where the inequality follows from the concavity of $\Phi_0$.
\end{proof}

However, in general, the uniform treatment proportions need not be optimal as shown in the following example.
\begin{example}
	Let $Q=(-1,1_{v-1}^T/(v-1))^T$ (i.e., in fact, we are examining $c$-optimality, see Chapter 2 by \cite{puk}), which aims at estimating the average comparison with control $\overline{\tau_i-\tau_0}$. Then, $\Psi_0(w) = w_1^{-1} + (v-1)^{-2} \sum_{i>1} w_i^{-1}$ for any $w>0$, and the unique optimal treatment design $w^*$ satisfies $w_1^* = 1/2$ and $w_i^* =(2(v-1))^{-1}$ for $i>1$, which obviously is not the uniform design.
\end{example}

Let us return to Example \ref{exSpecialContrasts}.

\begin{example}[Example \ref{exSpecialContrasts} cont.]
	Since the system of contrasts $Q_1$ in Example \ref{exSpecialContrasts} is a full-rank system of $v-1$ contrasts, the $D$-optimal design for $Q_1^T\tau$ is $\bar{w}=1_7/7$ and the corresponding graph $G_1$ has weight $\alpha_i=7$ on each vertex. Since the graph $G_1$ is a tree (in fact, $G_1$ is a tree rooted in vertex 1), any rooted spanning tree of $G_1$ is obtained by simply appropriately changing the directions of the edges in $G_1$. It follows that the design $\bar{w}$ minimizes the total weight of all rooted versions of $G_1$, weighted by the inverse design values.
\end{example}

\subsection{$A$-optimality}

Note that $\mathrm{tr}(\L_\alpha)=\sum_i d_i \alpha_i$, which can be expressed as $\sum_i \tilde{d}_i$, where $\tilde{d}_i := \alpha_id_i$ is the weighted degree of the vertex $i$. Then, the $A$-optimality value $\Psi_{-1}(w)=\mathrm{tr}(\L_w)$ is equal to the total weighted degree of graph $G$, i.e., the sum of all weighted degrees of its vertices, $\Psi_{-1}(w) = \sum_i \tilde{d}_i = \sum_i w_i^{-1}d_i$, where $\tilde{d}_i = w_i^{-1}d_i$. It generalizes the well-known fact that $\mathrm{tr}(L)$ is equal to twice the number of edges of $G$, or equivalently, to the total degree of $G$. As a consequence, an $A$-optimal treatment design $w$ minimizes the total weighted degree of $G$, i.e., the sum of weighted degrees of vertices in $G$, with weights inverse to the design values, $\alpha_i=w_i^{-1}$.

\begin{proposition}\label{pAoptPairwise}
	Let $Q^T\tau$ be a system of pairwise comparisons, let $w>0$ be a feasible design and let $G$ be the corresponding vertex-weighted graph. Then, $\Psi_{-1}(w) = \sum_i w_i^{-1} d_i = \sum_i \tilde{d}_i$ and $\Phi_{-1}(w) = r \big(\sum_i w_i^{-1} d_i)^{-1}$.
\end{proposition}

The value of the $A$-optimality criterion can be expressed without the graph terminology. Recall that we denote the elements of the coefficient matrix $Q$ as $q_{ij}$.

\begin{corollary}\label{cAopt}
	Let $Q^T\tau$ be a system of pairwise comparisons and let $w>0$ be a feasible design. Then, $\Psi_{-1}(w) = \sum_i w_i^{-1} \sum_j | q_{ij} |$ and  $\Phi_{-1}(w) = r  \big(\sum_i w_i^{-1} \sum_j | q_{ij} |\big)^{-1}$.
\end{corollary}

\begin{proof}
	The degree $d_i$ of a vertex $i$ is equal to the number of edges incident with $i$, which is the number of occurences of $\tau_i$ in $Q^T\tau$. Thus, $d_i = \sum_j |q_{ij}|$.
\end{proof}

Corollary \ref{cAopt} can be generalized to any system of contrasts resulting in a fairly trivial proposition; note that if $Q^T\tau$ is a system of pairwise comparisons, $|q_{ij}| = q_{ij}^2$ for each $i,j$.

\begin{proposition}\label{pAopt}
	Let $Q^T\tau$ be a system of treatment contrasts and let $w>0$. Then, $\Psi_{-1}(w) = \sum_i w_i^{-1} \sum_j q_{i,j}^2$ and $\Phi_{-1}(w) = r \big(\sum_i w_i^{-1} \sum_j q_{i,j}^2\big)^{-1}$.
\end{proposition}

\begin{proof}
	Let us calculate $\Psi_{-1}(w) = \mathrm{tr}(Q^TM^{-1}(w)Q) = \mathrm{tr}(M^{-1}(w)QQ^T) = \sum_i w_i^{-1} \sum_j q_{ij}^2$.
\end{proof}

From Propositions \ref{pAoptPairwise} and \ref{pAopt}, the $A$-optimal treatment proportions can be easily calculated.

\begin{proposition}
	Let $Q^T\tau$ be a system of contrasts. Then, the $A$-optimal treatment proportions for $Q^T\tau$ are given by
	\begin{equation}\label{eAopt}
	w_i^* = \frac{\sqrt{\sum_k q_{i,k}^2}}{\sum_j \sqrt{\sum_k q_{j,k}^2}}, \quad i=1,\ldots,v.
	\end{equation}
	In particular, if $Q^T\tau$ is a system of pairwise comparisons, the $A$-optimal treatment proportions are 
	\begin{equation}\label{eAoptPairwise}
	w_i^* = \frac{\sqrt{d_i}}{\sum_j\sqrt{ d_j}}, \quad i=1,\ldots,v,
	\end{equation}
	where $d_i$ is the degree of the $i$-th vertex in the corresponding graph $G$.
\end{proposition}

\begin{proof}
	It is straightforward to solve the optimization problem $\min \sum_i w_i^{-1} \sum_k q_{i,k}^2$, such that $\sum_i w_i=1$, analytically. It has a unique solution given by \eqref{eAopt}.
\end{proof}

Note that the $A$-optimal proportions \eqref{eAopt} can be obtained from Corollary 8.8 of \cite{puk}. The formula \eqref{eAoptPairwise} provides a straightforward interpretation of $A$-optimal proportions using the graph terminology. The $A$-optimal value for treatment $i$ is proportional to the square of the degree $d_i$ of the vertex $i$, i.e., the $A$-optimal proportions depend only on the number of times the particular treatments are present in the system $Q^T\tau$, the dependence being a square root. Interestingly, the $A$-optimal treatment proportions depend only on the \emph{local} properties of the graph (specifically, the numbers of neighbors of the vertices), not on the global structure of the graph.

\begin{example}[Example \ref{exSpecialContrasts} cont.]
	The results for the system of contrasts $Q_1^T\tau$ from Example \ref{exSpecialContrasts} are demonstrated in Figure \ref{fSpecialAopt}. In $G_1$, there are four vertices with degree $1$, one vertex with degree $2$ and two vertices with degree $3$. Thus, $\sum_j \sqrt{d_j} = 4+2\sqrt{3}+\sqrt{2}=:S$ and the $A$-optimal design values are $w_1^*=w_4^*=w_6^*=w_7^* = 1/S \approx 0.11$, $w_2=\sqrt{2}/S \approx 0.16$, $w_3^*=w_5^* = \sqrt{3}/S \approx 0.20$.
\end{example}

\begin{figure}[h]
	\centering
	\epsfig{file=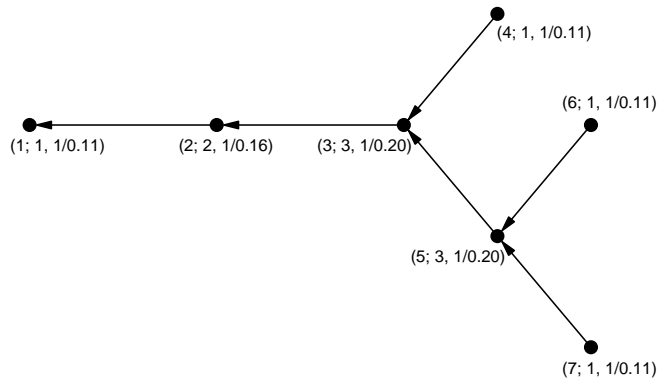, height=7cm}
	\vspace{-30pt}
	\caption{$A$-optimal design $w^*$ for system of contrasts $Q_1^T\tau$ from Example \ref{exSpecialContrasts}. The labels of the vertices are of the form $(i;d_i,\alpha_i)$, where $i$ is the vertex index, $d_i$ is the degree of vertex $i$, and $\alpha_i=1/w_i^*$ is the vertex weight for vertex $i$.}
	\label{fSpecialAopt}
\end{figure}

\subsection{$E$-optimality}

The quadratic form given by the Laplacian of the graph $x^TLx$ is a well-known expression for the unweighted or edge-weighted graphs
\begin{equation}\label{eLapQF}
x^TLx = \sum_{ij \in E} (x_i-x_j)^2, 
\end{equation}
where $ij$ is an edge connecting vertices $i$ and $j$.
It measures the energy of a physical system represented by a given graph, as noted by \cite{Mohar}.
In networks of dynamic agents, $x^TLx$ is denoted as the Laplacian potential of the graph, which measures the total 'disagreement' between agents in the network (see, e.g., \cite{SaberMurray}).

For vertex-weighted graphs, a similar expression of $x^T\L x$ can be established, which allows for expressing $\lambda_{\max}(\L)$ in a useful form (cf. (7.27) of \cite{CvetkovicEA} for normalized Laplacians).

\begin{proposition}
	Let $G=(V,E)$ be a graph with vertex weights $\alpha$. Then,
	$$\lambda_{\max}(\L) 
	= \max_{\sum_{i} x_i^2=1}  \sum_{(i,j) \in E} (\sqrt{\alpha_i}x_i - \sqrt{\alpha_j}x_j)^2
	=
	\max_{\sum_{i} \alpha_i^{-1}y_i^2=1} \sum_{(i,j) \in E} (y_i - y_j)^2
	$$
\end{proposition}

\begin{proof}
	The largest eigenvalue of $\L$ satisfies
	$$\begin{aligned}
	\lambda_{\max}(\L) 
	&= \max_{\lVert x \rVert = 1} x^T\L x
	= \max_{\sum_{i} x_i^2 = 1}  \sum_{j=1}^s \Big( \sum_{i=1}^v q_{ij} \sqrt{\alpha_i}x_i \Big)^2  \\
	&=\max_{\sum_{i} x_i^2 = 1}  \sum_{(i,j) \in E} (\sqrt{\alpha_i}x_i - \sqrt{\alpha_j}x_j)^2.
	\end{aligned}$$
	The second expression is obtained by setting $x_i=\alpha_i^{-1/2}y_i$.
\end{proof}

Using the derived expression for $\lambda_{\max}(\L)$, we may express the value of the $E$-optimality criterion.

\begin{corollary}
	Let $Q^T\tau$ be a system of pairwise comparisons, let $w>0$ be a feasible design and let $G$ be the corresponding vertex-weighted graph. Then,
	\begin{equation}\label{ePsiE1}
	\Psi_{-\infty}(w) 
	= \max_{\sum_{i} w_i y_i^2 = 1} \sum_{(i,j) \in E} (y_i - y_j)^2 = \max_{\lVert y \rVert_w=1} \sum_{(i,j) \in E} (y_i - y_j)^2,
	\end{equation}
	where $\lVert y \rVert_w:=(\sum_i w_iy_i^2)^{1/2}$,
	and 
	$$
	\Phi_{-\infty}(w)
	= \Big(\max_{\sum_{i} w_i y_i^2} \sum_{(i,j) \in E} (y_i - y_j)^2 \Big)^{-1}
	$$
\end{corollary}

The value of $\Psi_{-\infty}$ criterion has a straightforward interpretation. The expression $\sum_{(i,j) \in E} (y_i - y_j)^2 $ measures the total variability (potential, energy) of a function $y$ on vertices of $G$. Then, \eqref{ePsiE1} represents the maximum total variability over all $y$, which are normalized in the weighted norm $\lVert y \rVert_w$. Then, an $E$-optimal design minimizes the worst possible variability of $G$ over all vertex values $y$ with normalized weight $\lVert y \rVert_w$. Alternatively, it minimizes the worst possible total 'disagreement' of adjacent values of $y$ over all $y$, which are normalized with respect to the weighted norm $\lVert y \rVert_w$.


The $E$-optimality criterion does not seem to allow for such a direct formula for optimal weights, as \eqref{cAopt}. However, we provide $E$-optimal treatment proportions for a specific class of systems of pairwise comparisons. We will say that a system of pairwise comparisons $Q^T\tau$ is \emph{bipartite} if the corresponding graph $G$ is bipartite. It can be observed that a graph $G$ is bipartite if and only if there exists an orientation of $G$, such that each vertex $i \in V$ is either a sink (with zero vertices directed from $i$) or a source (with zero vertices directed towards $i$). Note that if $G$ is a tree, it is bipartite; thus, any full-rank system of $v-1$ pairwise comparisons is bipartite.

For proving $E$-optimality, we will employ the well-known Equivalence Theorem for $E$-optimality (Theorem 7.22 in \cite{puk}) provided below. By $\Xi$ we denote the set of the competing designs; $\Xi$ must either be the set of all feasible designs or some subset of all feasible designs.
\begin{lemma}
	Let $Q^T\tau$ be a full-rank system.
	A feasible design $w \in \Xi$ with its moment matrix $M$ and information matrix $N_Q$ is 
	$E$-optimal for estimating $Q^T \tau$ in $\Xi$ if and only if there exist a generalized inverse $G$ of $M$ and $E \in \mathfrak{S}_+^s$, $\mathrm{tr}(E)=1$, such that
	\begin{equation}\label{eGETE}
	\mathrm{tr}(M(\tilde{w})GKN_QEN_QK^T  G^T  ) \leq \lambda_{\min}(N_Q) \text{ for all } \tilde{w} \in \Xi.
	\end{equation}
\end{lemma}

However, since we consider also rank-deficient subsystems, we will slightly reformulate the Equivalence Theorem. As stated in Section 8.18 of \cite{puk}, the Equivalence Theorem for $\Phi_p$-optimality holds also in the rank deficient case, only instead of $N_Q(w)$, we have $C_Q(w)=(Q^TM^{-1}(w)Q)^+$ and instead of $\lambda_{\min}(N_Q(w))$, we have $1/\lambda_{\max}(V_Q(w))$). Moreover, if $E=hh^T$, where $h$ is a normalized eigenvector of $V_Q(w)$ corresponding to $\lambda_{\max}(V_Q(w))$, we obtain a simpler expression, which holds for both the full-rank as well as the rank deficient case.

\begin{lemma}
	Let $w \in \Xi$ be a feasible design for $Q^T\tau$ with its moment matrix $M$ and $V_Q:=V_Q(w)$, and let $h$ be an eigenvector of $V_Q$ corresponding to $\lambda_{\max}(V_Q)$, satisfying $\lVert h \rVert = 1$. Then, if there exists a generalized inverse $G$ of $M$, such that
	\begin{equation}\label{eGETE2}
	h^T Q^TG^TM(\tilde{w}) GQh \leq \lambda_{\max}(V_Q) \text{ for all } \tilde{w} \in \Xi,
	\end{equation}
	then $w$ is $E$-optimal for $Q^T\tau$ in $\Xi$.
\end{lemma}

\begin{proof}
	Let $E=hh^T$.
	Suppose that $Q^T\tau$ is a full-rank system. From $N_Qh = \lambda_{\max}^{-1}(V_Q)h$, it follows that the left-hand side of \eqref{eGETE} is equal to $\lambda_{\max}^{-2}(V_Q) h^T Q^TG^TM(\tilde{w}) GQh$. Then, \eqref{eGETE} can be rearranged to \eqref{eGETE2}.
	
	If $Q^T\tau$ is rank deficient, the normality inequality of the Equivalence Theorem becomes $\mathrm{tr}(M(\tilde{w})GKV^+EV^+K^T  G^T  ) \leq 1/\lambda_{\max}(V_Q) \text{ for all } \tilde{w} \in \Xi.$ Since $h$ is an eigenvector of $V_Q$ corresponding to $\lambda_{\max}(V_Q)$, it is an eigenvector of $V_Q^+$ corresponding to $1/\lambda_{\max}(V_Q)$, which yields the same inequality as in the full-rank case.
\end{proof}

\begin{theorem}\label{tEopt}
	Let $Q^T\tau$ be a bipartite system of pairwise comparisons and let $G$ be the corresponding graph. Then, the treatment proportions
	\begin{equation}\label{eEopt}
	w_i^*=\frac{d_i}{\sum_j d_j}, \quad i=1,\ldots,v
	\end{equation}
	are $E$-optimal for $Q^T\tau$ with the optimal value $\Psi_{-\infty}(w^*)=\lambda_{\max}(V(w^*))=2\sum_i d_i = 4s$. Moreover, if each vertex of $G$ is either a sink or a source, then $h=1_s/\sqrt{s}$ is an eigenvector of $V_Q(w^*)$ corresponding to $\lambda_{\max}(V_Q(w^*))$.
\end{theorem}

\begin{proof}
	Suppose that each vertex of $G$ is either a sink or a source. If it were not, we could change the orientations of the edges, which does not affect the Laplacian and thus it does not affect $\lambda_{\max}(V_Q)$ either.
	We denote $w:=w^*$, $k:=\sum_j d_j$ and $V_Q:=V_Q(w)$. First, we calculate $Qh = Q1_s/\sqrt{s} = g/\sqrt{s}$, where $g_i$ is either $d_i$ or $-d_i$, which we denote as $g_i = \pm d_i$, $i=1,\ldots,v$. Then,
	$$V_Qh = Q^TM^{-1}(w)g/\sqrt{s} = 2k1_s/\sqrt{s} = 2k h,$$
	because $w_i^{-1}g_i = \pm k$ for all $i$.
	Hence, $h$ is an eigenvector of $V_Q$ corresponding to $\lambda^*:=2k=4s$. Now, we will prove that $\lambda^* = \lambda_{\max}(V_Q)$.
	
	The $k,\ell$-th element of $V_Q$, $v_{k,\ell}$, satisfies $v_{k,\ell} = \sum_i w^{-1}q_{i,k}q_{i,\ell}$. The graph representation yields that the indices $k$ and $\ell$ represent edges and provides a formula for $v_{k,\ell}$. For $k=\ell$, we have
	$$v_{k,k} = 
	w_i^{-1}+w_j^{-1}, \text{ where } k=(i,j)
	$$
	and since every vertex is a sink or a source, for $k\neq\ell$ we obtain
	$$v_{k,\ell} = \begin{cases}
	w_i^{-1}, &\text{there exists a vertex $i$, which is incident with both $k$ and $\ell$,} \\
	0, &\text{otherwise.}
	\end{cases} $$
	Then,
	$\lambda_{\max}(V_Q) = \max_{\lVert x \rVert = 1} x^T V_Q(w) x$. Let us alternatively denote the edges by the vertices they connect, i.e., the edge $e=(i,j)$ will be denoted as an unordered pair $ij$. Similarly, we index the elements of $x$ in $x^T V_Q x$ by the corresponding pairs of vertices. Then,
	$$ x^T V_Q x = \sum_{i_1j_1 \in E} \sum_{i_2j_2 \in E} v_{i_1j_1,i_2j_2}x_{i_1j_1} x_{i_2j_2} = \sum_{i=1}^v w_i^{-1} \sum_{p \sim i} \sum_{q \sim i} x_{ip} x_{i q} = k \sum_{i=1}^v d_i^{-1} \sum_{p \sim i} \sum_{q \sim i} x_{ip} x_{i q} $$
	for any $x \in \R^s$. Let us denote as $n(i)$ the set of all vertices adjacent to $i$, $i=1,\ldots,v$. Then,
	$$\lambda_{\max}(V_Q) = x^T V_Q x = \max_{\lVert x \rVert = 1} k \sum_{i=1}^v d_i^{-1} \big(\sum_{j \in n(i)} x_{ij}\big)^2.$$
	The Cauchy-Schwarz inequality yields $(\sum_{j \in n(i)} x_{ij})^2 \leq d_i \sum_{j \in n(i)} x_{ij}^2$ and thus 
	$$\lambda_{\max}(V_Q) \leq k \max_{\lVert x \rVert = 1} \sum_{i=1}^v \sum_{j \in n(i)} x_{ij}^2 = 2k \max_{\lVert x \rVert = 1} \sum_{(i,j)\in E} x_{ij}^2 = 2k = \lambda^*.$$
	It follows that $\lambda^* = \lambda_{\max}(V_Q)$.
	
	Let $E=hh^T$ and $G=M^{-1}(w)$ and recall that $Qh = g /\sqrt{s}$, where $g_i=\pm d_i$. Then, the left-hand side of \eqref{eGETE2} is
	$$\frac{1}{s} g^T \diag(\tilde{w}_1w_1^{-2}, \ldots, \tilde{w}_v w_v^{-2})g = \frac{1}{s} k^2 \sum_{i=1}^v \tilde{w}_i d_i^{-2} g_i^2 = \frac{k^2}{s} \sum_{i=1}^v \tilde{w}_i = \frac{k^2}{s} = 4s, $$
	which is equal to the right-hand side. Thus, $w$ is $E$-optimal.
\end{proof}

For the subclass of bipartite systems of pairwise comparisons, we obtained an analytical expression for $E$-optimal treatment proportions, similar to the formula \eqref{eAopt} for $A$-optimality. Only instead of square roots of the vertex degrees, the $E$-optimal treatment designs are proportional directly to the vertex degrees, i.e., to the numbers of times the corresponding treatments are present in $Q^T\tau$. 

Similar to the $A$-optimal treatment proportions, the $E$-optimal proportions for bipartite systems depend only on the local properties of the graph - the numbers of the neighbors of the vertices. However, for general systems of pairwise comparisons, this need not be true. The global properties of the graphs are embedded in the assumption that the graph is bipartite. In the presence of cycles, the simple dependence of the $E$-optimal designs only on the degrees of the vertices need not hold.

Note that if the bipartite graph $G$ does not consist of only sinks and sources, a normalized eigenvector $h$ corresponding to $\lambda_{\max}(V_Q(w^*))$ can be obtained as follows. The elements of $h$ belong to the set $\{-1/\sqrt{s},1/\sqrt{s}\}$ and for any vertex $i$ the following holds: the signs of all elements of $h$ corresponding to the edges directed from $i$ are the same, and they are opposite to the signs corresponding to the edges directed to $i$. Such eigenvector $h$ can be constructed by arbitrarily choosing one value $h_k \in \{-1/\sqrt{s},1/\sqrt{s}\}$ and then iteratively obtaining the signs of the incident edges. The choices $h_k=-1/\sqrt{s}$, $k \in K \subseteq \{1,\ldots, s\}$, can be thought of as reversing the direction of edges $k \in K$ in order to convert all vertices to sinks or sources; the choices $h_k =1/\sqrt{s}$ mean maintaining the original directions of the corresponding edges.

\begin{example}[Example \ref{exSpecialContrasts} cont.]
	The $E$-optimal design $w^*$ for the system of contrasts $Q_1^T\tau$ from Example \ref{exSpecialContrasts} is given in Figure \ref{fSpecialEopt}. The total degree of $G_1$ is $\sum_j d_j = 12$ and thus the particular design values are $w_1^*=w_4^*=w_6^*=w_7^*=1/12$, $w_2^*=1/6$ and $w_3^*=w_5^*=1/4$. The eigenvector $h$ of $V_Q(w^*)$ corresponding to $\lambda_{\max}(V_Q(w^*))=2\sum_j d_j = 24$ is $h=(1,-1,1,1,-1,-1)^T/\sqrt{6}$, which is also represented in Figure \ref{fSpecialEopt} by values $+1$ or $-1$ on the edges. Notice that by reversing the directions of all edges for which $h_k=-1/\sqrt{6}$, all vertices of $G_1$ become sinks or sources.
\end{example}

\begin{figure}[h]
	\centering
	\epsfig{file=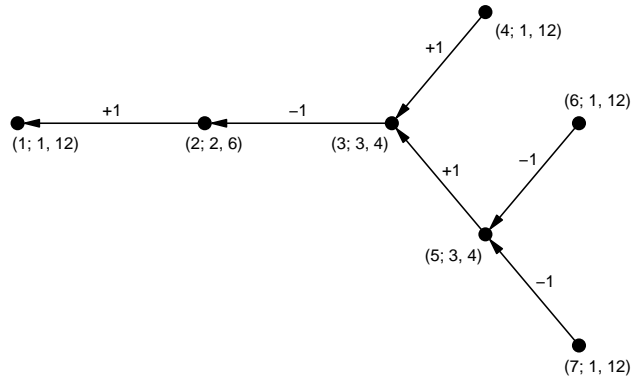, height=7cm}
	\vspace{-30pt}
	\caption{$E$-optimal design $w^*$ for the system of contrasts $Q_1^T\tau$ from Example \ref{exSpecialContrasts}. The labels of the vertices are of the form $(i;d_i,\alpha_i)$, where $i$ is the vertex index, $d_i$ is the degree of vertex $i$, and $\alpha_i=1/w_i^*$ is the vertex weight for vertex $i$. The values $+1$ or $-1$ on the edges represent the signs of the respective elements of the eigenvector $h$ of $V_Q(w^*)$ corresponding to $\lambda_{\min}(V_Q(w^*))$.}
	\label{fSpecialEopt}
\end{figure}

Note that the treatment proportions given by Theorem \ref{tEopt} need not be optimal for general systems of pairwise comparisons, as demonstrated in the following example.

\begin{example}
	Let $v=4$ and consider estimating $\tau_1-\tau_2$, $\tau_2-\tau_3$, $\tau_3-\tau_1$ and $\tau_1-\tau_4$. The degrees of the vertices in the corresponding graph $G$ are $d_1 = 3$, $d_2=2$, $d_3=2$ and $d_4=1$. The treatment proportions given by \eqref{eEopt} are $w=(3/8,1/4,1/4,1/8)^T$ and it can be calculated that $\lambda_{\max}(V_Q(w)) \approx 13.8297$. For $\tilde{w}=(0.38,0.23,0.23,0.16)^T$, it can be calculated that $\lambda_{\max}(V(\tilde{w})) \approx 13.0435< \lambda_{\max}(V_Q(w))$. Thus, $w$ is not $E$-optimal.
\end{example}

\section{Symmetric systems of contrasts}\label{sSymmetric}

If a system of contrasts is 'symmetric', intuitively, the uniform design $\bar{w}=1_v/v$ should be optimal with respect to a wide range of optimality criteria. Using the graph representation, we may obtain such symmetric systems of contrasts.

In Section \ref{ssPerm} we noted that if $\pi$ is an automorphism, the matrix $\L_{P_\pi w}$ is orthogonally similar to $\L_w$. Therefore, Theorem \ref{tEigLap} yields that $\Phi(P_\pi w) = \Phi(w)$ for any orthogonally invariant $\Phi$. This observation is a useful tool in proving the optimality of $\bar{w}$ for a 'symmetric' system of contrasts.

\begin{lemma}\label{lLapPerm}
	Let $w$ be a feasible design for estimating a system of pairwise comparisons $Q^T\tau$ and let $\pi$ be an automorphism of the corresponding graph $G$. Then,
	\begin{equation}\label{eLapPerm}
	\mathcal{L}_{P_\pi w} = P_\pi \mathcal{L}_w P_\pi^T.
	\end{equation}
	and $\Phi(P_\pi w) = \Phi(w)$ for any orthogonally invariant information function $\Phi$.
\end{lemma}

The following theorem shows that if a cyclic permutation is an automorphism of $G$, the uniform treatment design is $\Phi$-optimal for $Q^T\tau$ for any orthogonally invariant information function. Recall that a permutation is cyclic if it consists of only one cycle.

\begin{theorem}\label{tCyclic}
	Let $Q^T\tau$ be a system of pairwise comparisons $\tau_i-\tau_j$ and let $G$ be the corresponding graph. Suppose that there exists a cyclic permutation $\pi$, which is an automorphism of $G$. Then, $\bar{w}=1_v/v$ is  $\Phi$-optimal for estimating $Q^T\tau$ with respect to any orthogonally invariant information function $\Phi$.
\end{theorem}

\begin{proof}
	Let $i \in 1,\ldots,v-1$, and let $w>0$ be a feasible treatment proportions design. Note that if $\pi$ is an automorphism, $\mathcal{L}_{P_\pi^i w} = P_\pi^i \mathcal{L}_w (P_\pi^i)^T$ and thus $\Phi(P_\pi^i w) = \Phi(w)$ holds for any $i\in\mathbb{N}$.
	
	Since $\pi$ is a cyclic permutation, the uniform treatment design can be expressed as 
	$$\bar{w} = \frac{1}{v}\sum_{i=0}^{v-1} P_\pi^i w$$
	and hence,
	$$\begin{aligned}
	\Phi\big(\bar{w}\big)
	&=
	\Phi\left(\frac{1}{v}\sum_{i=0}^{v-1} P_\pi^iw\right) \geq
	\frac{1}{v}\sum_{i=0}^{v-1} \Phi(P_\pi^iw)=  \\
	&= \frac{1}{v}\sum_{i=0}^{v-1} \Phi(w) = \frac{1}{v}v \Phi(w) = \Phi(w),
	\end{aligned}$$
	where the inequality follows from the concavity of $\Phi$.
	Thus, $\bar{w}$ is $\Phi$-optimal.
\end{proof}

Note that for \eqref{eLapPerm} to hold, it is sufficient for $\pi$ to satisfy $P_\pi A P_\pi^T = A$, or equivalently, $P_\pi RR^T P_\pi^T = RR^T.$
That is, it is not necessary to preserve the orientation of the edges. It follows from the fact that the criterial value $\Phi$ for any orthogonally invariant function $\Phi$ is determined by the eigenvalues of the weighted Laplacian, which does not depend on the orientation of the edges.

Furthermore, note that the proof of Theorem \ref{tCyclic} does not employ the condition that $Q^T\tau$ is a system of pairwise comparisons. Hence, we may formulate a theorem for general systems of contrasts. We say that a system of contrasts $Q^T\tau$ (not necessarily a system of pairwise comparisons) is \emph{cyclic} if there exists a cyclic permutation $\pi$ satisfying
\begin{equation}\label{ePerm}
P_\pi QQ^T P_\pi^T = QQ^T.
\end{equation}

\begin{theorem}\label{tOrthOptGen}
	Let $Q^T\tau$ be a cyclic system of contrasts. Then, $\bar{w}=1_v/v$ is $\Phi$-optimal for estimating $Q^T\tau$ with respect to any orthogonally invariant information function $\Phi$.
\end{theorem}

\begin{proof}
	The proof is analogous to the proof of Theorem \ref{tCyclic}; the matrix $\mathcal{L}_w$ is not considered to be the Laplacian matrix of the corresponding graph, but formally, $\mathcal{L}_w:=M^{-1/2}(w)QQ^TM^{-1/2}(w)$ for any coefficient matrix $Q$ and any $w>0$. Similarly as in Theorem \ref{tEigLap}, the value $\Phi(w)$ is determined by the eigenvalues of $\L_w$.
\end{proof}

Note that Theorem 5 of \cite{RosaHarman16}, which states that if $QQ^T$ is completely symmetric, then $\bar{w}$ is $\Phi$-optimal for any orthogonally invariant criterion, is a corollary of Theorem \ref{tOrthOptGen} here. If $QQ^T$ is completely symmetric, i.e., if $QQ^T = aI_v +bJ_v$ for some $a,b$, then \eqref{ePerm} is satisfied for any permutation matrix and thus for any cyclic permutation.
However, the completely symmetric systems do not cover all systems satisfying Theorem \ref{tOrthOptGen}, as demonstrated in the following examples.

\begin{example}\label{exTwoSets}
	Let $Q^T\tau$ be a system of contrasts for comparing two sets of treatments of equal size (say $g$), i.e., $\tau_j-\tau_i$, $i=1,\ldots,g$, $j=g+1,\ldots,2g$. This is a special case of \emph{comparing $v-g$ treatments with $g$ controls}, where $v-g=g$, see, e.g., \cite{Majumdar86} or \cite{GithinjiJacroux}.
	
	Then, $Q=(-I_g \otimes 1_{g}, 1_{g} \otimes I_{g})^T$ and 
	$$
	QQ^T = \begin{bmatrix}
	gI_g & -J_g \\ -J_g & gI_g
	\end{bmatrix}.
	$$
	Let $\pi$ be the cyclic permutation $\pi: \pi(i) = g+i$ for $i\leq g$ and $\pi(i) = i-g+1$ for $i>g$, i.e., $\pi$ can be represented by the cycle  $(1,g+1,2,g+2,3,g+3,\ldots,g,2g)$ (see Figure \ref{fCwCsymm}).
	It can be verified that $P_\pi QQ^T P_\pi^T = QQ^T$ and therefore, $\bar{w}$ is $\Phi$-optimal for estimating $Q^T\tau$ with respect to any orthogonally invariant criterion.
\end{example}

\begin{example}\label{exAllPairwise}
	Let $Q^T\tau$ be a system of contrasts $\tau_2-\tau_1$, $\tau_3-\tau_2, \ldots, \tau_v-\tau_{v-1}$, $\tau_1-\tau_v$. Then, $Q^T$ is given by rotations of its first row $(-1,1,0_{v-2}^T)$ and similarly, $QQ^T$ is given by rotations of its first row $(2,-1,0_{v-3}^T,-1)$ 
	$$
	Q^T=\begin{bmatrix}
	-1 & 1 & 0 & \ldots & 0 \\
	0 & -1 & 1 & \ldots & 0 \\
	\vdots & & & & \vdots \\
	0 & \ldots & 0 & -1 & 1 \\
	1 & 0 & \ldots & 0 & -1
	\end{bmatrix},\quad
	QQ^T\begin{bmatrix}
	2 & -1 & 0 & \ldots & 0 & -1 \\
	-1 & 2 & -1 & 0 & \ldots & 0 \\
	\vdots & & & & & \vdots \\
	0 & \ldots & 0 & -1 & 2 & -1 \\
	-1 & 0 & \ldots & 0 & -1 & 2
	\end{bmatrix}.
	$$
	Clearly, $QQ^T$ satisfies \eqref{ePerm}, where $\pi=(1,2,3,\ldots,v)$, i.e., $\pi(i)=i+1$ for $i<v$ and $\pi(v)=1$ (see Figure \ref{fCyclic}) and thus $\bar{w}$ is $\Phi$-optimal for estimating $Q^T\tau$ with respect to any orthogonally invariant criterion.
\end{example}

\begin{figure}[h]
	\centering
	\begin{subfigure}[b]{0.4\textwidth}
		\epsfig{file=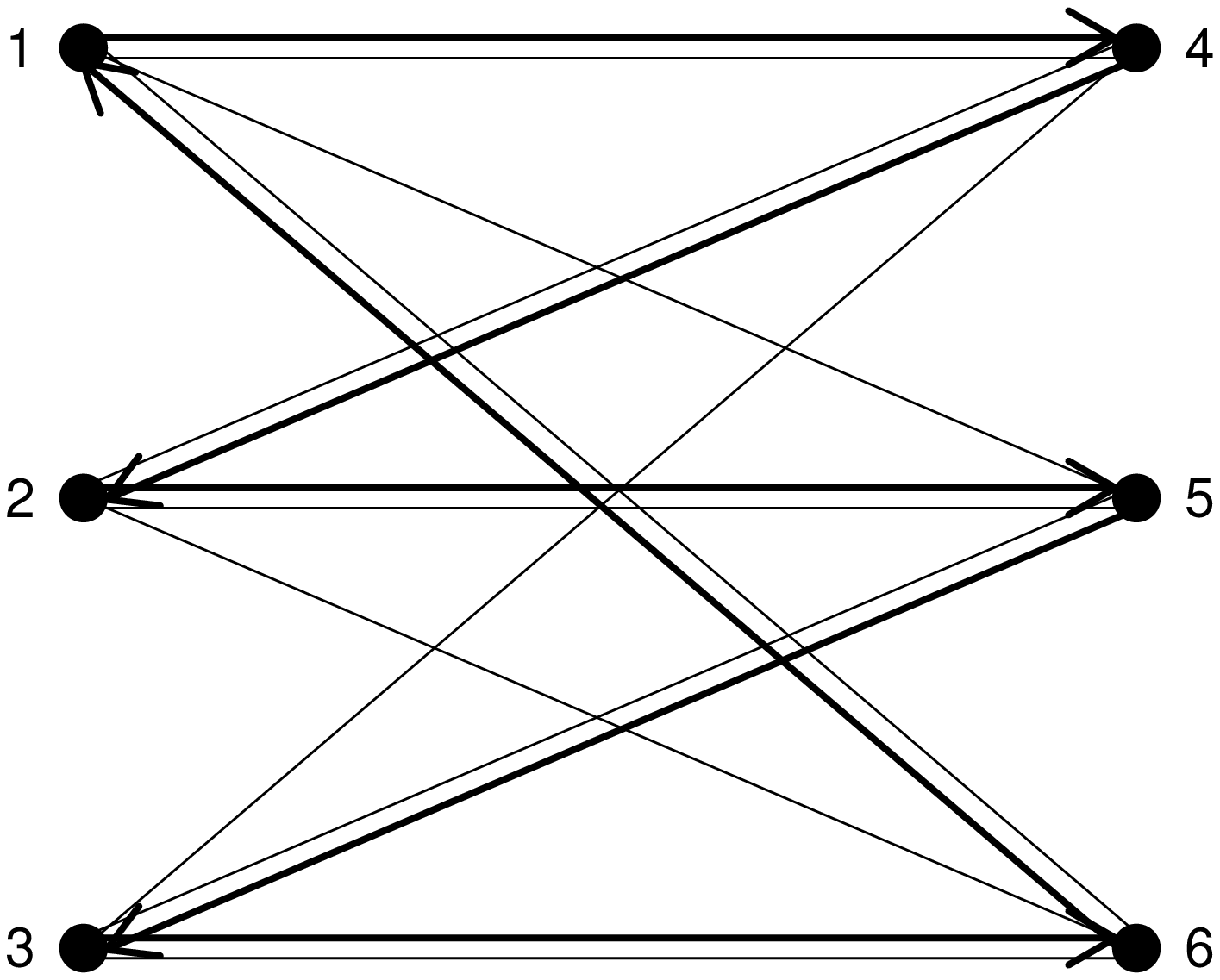, height=7cm}
		\vspace{-40pt}
		\caption{Comparing 3 treatments with 3 controls}
		\label{fCwCsymm}
	\end{subfigure}
	~ 
	\begin{subfigure}[b]{0.4\textwidth}
		\epsfig{file=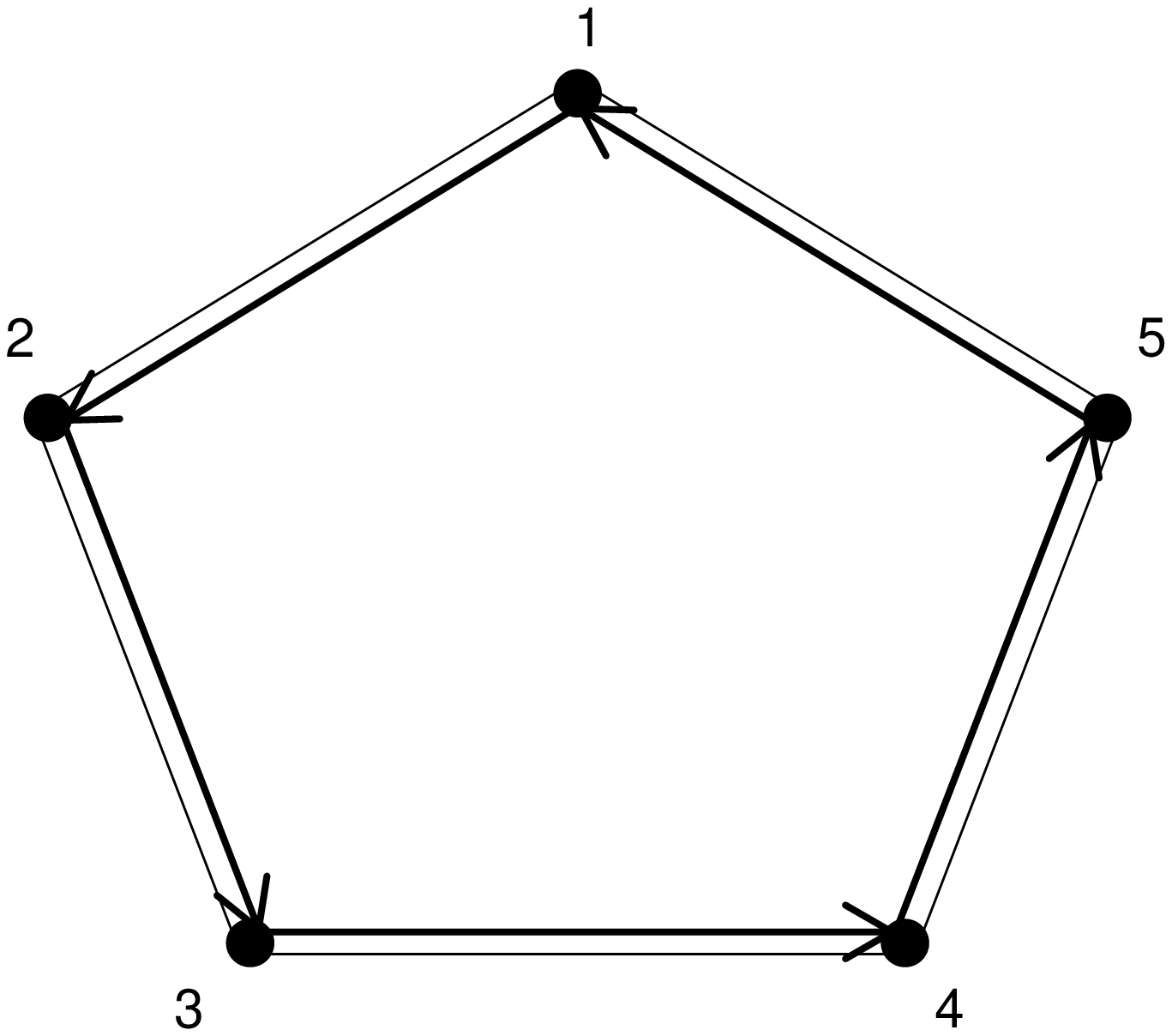, height=7cm}
		\vspace{-40pt}
		\caption{Cyclic comparisons \\ \phantom{a}}
		\label{fCyclic}
	\end{subfigure}
	\caption{Graph representations of systems of contrasts; the directions of the edges are suppressed as they do not affect the results. The cyclic automorphisms are represented by bold arrows.}\label{fSymmetric}
\end{figure}

If $\pi$ is an arbitrary permutation consisting of cycles $c_1, \ldots, c_K$ that satisfies \eqref{ePerm}, the proof of Theorem \ref{tOrthOptGen} can be replicated, resulting in analogous results for each cycle $c_i$.

\begin{theorem}\label{tOptSets}
	Let $Q^T\tau$ be a system of contrasts and let $\pi$ be a permutation satisfying \eqref{ePerm} that consists of cycles $c_1, \ldots, c_K$, i.e., $\pi=c_1 \ldots c_K$. Let $\Phi$ be an orthogonally invariant information function. Then, there exists a $\Phi$-optimal design $w$ that satisfies $w(i)=w(j)$ for all $i,j \in c_k$ for all $k \in \{1,\ldots,K\}$.
\end{theorem}

Note that in the case of a strictly concave criterion $\Phi$, the conditions of Theorem \ref{tOptSets} are also necessary conditions of optimality; similarly for Theorem \ref{tOrthOptGen}. That is, any $\Phi$-optimal design must be uniform on each of the given cycles $c_k$. 

Theorem \ref{tOptSets} can be used to simplify the search for optimal treatment proportions. For example, let $\Phi$ be orthogonally invariant and consider comparing a set of test treatments with a set of controls. Then, to find a $\Phi$-optimal treatment design $w$, it is sufficient to consider only two treatment proportions: one for the test treatments and one for the controls, as shown in Example \ref{exCwControls}. For the Kiefer's optimality criteria $\Phi_p$, these optimal treatment proportions are given in Theorem 6 of \cite{RosaHarman16}.

\begin{example}\label{exCwControls}
	Let $Q^T\tau$ be a system of contrasts for comparing $v-g$ treatments with $g$ controls, $g<v/2$, i.e., $\tau_j-\tau_i$, $i=1,\ldots,g$, $j=g+1,\ldots,v$. Then, $Q=(-I_g \otimes 1_{v-g}, 1_{g} \otimes I_{v-g})^T$ and
	$$
	QQ^T = \begin{bmatrix}
	(v-g)I_g & -J_{g \times (v-g)} \\ -J_{(v-g) \times g} & gI_{v-g}
	\end{bmatrix}.
	$$
	Let $\pi:(1,\ldots,g)(g+1,\ldots,v)$, i.e., $\pi(i)=i+1$ for $i \neq g, i \neq v$; $\pi(g)=1$ and $\pi(v)=g+1$. Then, $\pi$ satisfies \eqref{ePerm}. It follows that for any orthogonally invariant criterion $\Phi$ there exists a $\Phi$-optimal design satisfying $w(i)=\gamma$ for all $i \in \{1,\ldots,g\}$ and $w(j) = (1-g\gamma)/(v-g)$ for all $j \in \{g+1,\ldots,v\}$, for some $\gamma \in (0,1/g)$. 
\end{example}

\bibliographystyle{plainnat}
\bibliography{C:/BibTeX/rosa.bib}

\end{document}